\documentclass{article}
\usepackage[margin=1in]{geometry}
\usepackage{amsmath, amsthm, amssymb}
\usepackage{mathrsfs}
\usepackage{color}
\usepackage{enumerate}
\usepackage{url}
\usepackage{tikz}
\usepackage[capitalize]{cleveref}
\usepackage{graphicx}

\newcommand{\wt}{{\rm wt}}
\newcommand{\Ker}{{\rm Ker}}
\newcommand{\GF}{{\rm GF}}
\newcommand{\Z}{{\mathbb Z}}
\newtheorem{thm}{Theorem}[section]
\newtheorem{lemma}[thm]{Lemma}
\newtheorem{cor}[thm]{Corollary}
\newtheorem{prop}[thm]{Proposition}
\newtheorem{defn}[thm]{Definition}
\newtheorem{example}[thm]{Example}
\newtheorem{question}[thm]{Question}

\allowdisplaybreaks[4]

\long\def\symbolfootnote[#1]#2{\begingroup%
\def\thefootnote{\fnsymbol{footnote}}\footnote[#1]{#2}\endgroup}

\definecolor{mygreen}{rgb}{0.25,0.6,0.25}

\begin{document}

\title{Constructions of difference sets in nonabelian 2-groups}
\author{
Taylor Applebaum\footnote{University of Richmond VA, now at Google, {\tt applebaum.taylor@gmail.com}}, \quad
John Clikeman\footnote{University of Richmond VA, now at Google, {\tt jclikeman@gmail.com}}, \quad
James A. Davis\footnote{University of Richmond VA, {\tt jdavis@richmond.edu}; supported by NSA grant H98230-12-1-0243}, \\
John F. Dillon\footnote{National Security Agency, Ft.\ George G.\ Meade, MD 20755, {\tt jfdillon@gmail.com}}, \quad
Jonathan Jedwab\footnote{Department of Mathematics, Simon Fraser University, 8888 University Drive, Burnaby, BC V5A 1S6, Canada, {\tt jed@sfu.ca}; supported by NSERC}, \quad
Tahseen Rabbani\footnote{Department of Computer Science, University of Maryland MD, {\tt trabbani@cs.umd.edu}}, \\
Ken Smith\footnote{Department of Mathematics and Statistics, Sam Houston State University, Huntsville TX, {\tt kenwsmith54@gmail.com}}, \quad
William Yolland\footnote{Simon Fraser University, now at MetaOptima Technology Inc., {\tt william@metaoptima.com}}
}

\date{March 27, 2020 (revised January 12, 2022)}

\maketitle
\vspace{-10mm}
\begin{center}
\emph{Dedicated to the memory of Robert A. Liebler, a friend and mentor, and a passionate advocate \\for studying the action of finite nonabelian groups on combinatorial designs.}
\end{center}

\begin{abstract}
Difference sets have been studied for more than 80 years. 
Techniques from algebraic number theory, group theory, finite geometry, and digital communications engineering have been used to establish constructive and nonexistence results. 
We provide a new theoretical approach which dramatically expands the class of $2$-groups known to contain a difference set, by refining the concept of covering extended building sets introduced by Davis and Jedwab in 1997.
We then describe how product constructions and other methods can be used to construct difference sets in some of the remaining $2$-groups. 
We announce the completion of ten years of collaborative work to determine precisely which of the 56,092 nonisomorphic groups of order 256 contain a difference set. 
All groups of order 256 not excluded by the two classical nonexistence criteria are found to contain a difference set, in agreement with previous findings for groups of order 4, 16, and 64.
We provide suggestions for how the existence question for difference sets in 2-groups of all orders might be resolved.
\end{abstract}

\symbolfootnote[0]{2010 Mathematics Subject Classification 05B10, 05E18 (primary)}

\section{Motivation and Overview}
\label{sec:intro}

Difference sets were introduced by Singer~\cite{singer} in 1938 as regular automorphism groups of projective geometries. These examples are contained in the multiplicative group of a finite field, and hence the difference sets in those geometric settings occur in cyclic groups. In the decades following, difference sets were discovered in other abelian groups and subsequently in nonabelian groups. 
The central objective is to determine which groups contain at least one difference set. 
Researchers have developed a range of techniques in pursuit of this objective, taking advantage of connections with design theory, coding theory, cryptography, sequence design, and digital communications.

A $k$-subset $D$ of a group $G$ of order $v$ is a {\it difference set} with parameters $(v, k, \lambda)$ if, 
for all nonidentity elements $g$ in $G$, the equation 
$$
xy^{-1} = g 
$$
has exactly $\lambda$ solutions $(x,y)$ with $x, y \in D$; the related parameter $n$ is defined to be $k-\lambda$.  
The complement of a difference set with parameters $(v,k,\lambda)$ is itself a difference set, with parameters $(v, v-k, v-2k+\lambda)$ and the same related parameter~$n$.
The difference set is nontrivial if $1 < k < v-1$.
A $(v,k,\lambda)$ difference set in $G$ is equivalent to a symmetric $(v,k,\lambda)$ design with a regular automorphism group~$G$~\cite{bjl}.

Given an element $A=\sum_{g \in G} a_g g$ in the group ring $\Z G$, where each $a_g \in \Z$, 
we write $A^{(-1)}$ for the element $\sum_{g \in G} a_g g^{-1}$. 
It is customary in the study of difference sets to abuse notation by identifying a subset $D$ of a group $G$ with the element of the group ring $\Z G$ which is its $\{0,1\}$-valued characteristic function.
The subset $D$ of $G$ is then a difference set if and only if the $\{0,1\}$-valued characteristic function $D$ satisfies the equation
$$
D D^{(-1)} = n + \lambda G \quad \mbox{in $\Z G$},
$$
in which $n$ represents $n 1_G$.
Throughout, we shall instead identify the subset $D$ of $G$ with the element of $\Z G$ which is its $\{\pm 1\}$-valued characteristic function (taking the value $-1$ for each element of $G$ in $D$, and $+1$ for each element of $G$ not in~$D$). Under this convention, the subset $D$ of $G$ is a difference set if and only if the $\{\pm 1\}$-valued function $D$ satisfies
$$
D D^{(-1)} = 4n + (v - 4n)G \quad \mbox{in $\Z G$}.
$$
When $v=4n$, this reduces to
\begin{equation} \label{eqn:DDstar}
D D^{(-1)} = |G|,
\end{equation}
in which case the subset $D$ is called a {\it Hadamard} difference set because the $\{\pm 1\}$-valued $v \times v$ incidence matrix, whose rows and columns are indexed by the elements of $G$ and whose $(g,h)$ entry is the coefficient of $g^{-1}h$ in $D$, is a Hadamard matrix.  

\begin{example}[Bruck 1955 \cite{rhb1}]
\label{ex:bruck}
Let $G = C_2^4 = \langle x_1, x_2, x_3, x_4\rangle$, where $C_2$ denotes the multiplicative cyclic group of order~$2$. The set 
\[
D = \{1,\, x_1,\, x_2,\, x_3,\, x_4,\, x_1x_2x_3x_4\}
\]  
is a $(16,6,2)$ Hadamard difference set in~$G$.
We identify this set with the element 
$D = -1 - x_1 - x_2 - x_3 - x_4 - x_1x_2x_3x_4 
 + x_1 x_2 + x_1 x_3 + x_1 x_4 + x_2 x_3 + x_2 x_4 + x_3 x_4 + x_1 x_2 x_3 + x_1 x_2 x_4 + x_1 x_3 x_4 + x_2 x_3 x_4$ of the group ring~$\Z G$, and then $D D^{(-1)} = 16$.
\end{example}

We call a group containing a Hadamard difference set a \emph{Hadamard group}, and denote the class of Hadamard groups by~${\cal H}$. 
It is an outstanding problem in combinatorics to determine which groups belong to the class ${\cal H}$; see \cite{davisjedwabsurvey} for a survey and \cite{jungnickel-survey-update2} for a summary of subsequent results.
This paper focusses on determining which 2-groups (namely groups whose order is a power of~2) belong to~${\cal H}$.
The relation $v = 4n$ between the parameters of a difference set forces the parameters to be
\begin{equation} \label{eqn:had_params}
(v, k, \lambda) = (4N^2, 2N^2-N, N^2-N)
\end{equation}
for some integer $N$ \cite{menon2}.
Here $N$ can be positive or negative, and the two values $\pm N$ give the parameters of complementary difference sets and designs.
A nontrivial difference set in a 2-group must also have parameters of the form~\eqref{eqn:had_params}, where $N = 2^d$ for some positive integer~$d$ \cite{mann}. We therefore restrict attention to the parameters 
\[
(v, k, \lambda) = (2^{2d+2},2^{2d+1}-2^d,2^{2d}-2^d),
\]
where $d$ is a nonnegative integer. 
The groups of order $2^{2d+2}$ form a rich source of potential Hadamard difference sets:
there are 2 nonisomorphic groups of order 4 (both of which contain a trivial Hadamard difference set); 14 of order 16; 267 of order 64; 56,092 of order~256; and 49,487,367,289 groups of order 1024~\cite{besche-eick-obrien,burrell,oeis-order2n}.

The following product construction contains, as a special case, the earlier result \cite{menon2, turyn} that the class ${\cal H}$ is closed under direct products.

\begin{thm}[Dillon product construction 1985 \cite{dillon}]
\label{thm:productconstruction}
Suppose that $H_1, H_2 \in {\cal H}$, and that $G$ is a group containing subgroups $H_1$ and $H_2$ satisfying $G=H_1H_2$ and $H_1 \cap H_2=1$.
Then $G \in {\cal H}$. 
\end{thm}

\begin{proof}
Let $D_1$ and $D_2$ be difference sets in $H_1$ and $H_2$, respectively, and let $D=D_1D_2$.  
By hypothesis, every element $g$ of $G$ has a unique representation $g=h_1h_2$ for some $h_1 \in H_1$ and $h_2 \in H_2$, and so $D$ is $\{\pm 1\}$-valued. 
Then 
\[
DD^{(-1)} = (D_1D_2)(D_1D_2)^{(-1)} =  D_1D_2 D_2^{(-1)}D_1^{(-1)} =  D_1|H_2|D_1^{(-1)} =  |H_1||H_2| = |G|.
\]
\end{proof}
 
In a seminal paper, Turyn used algebraic number theory to prove a first nonexistence result for Hadamard $2$-groups.
 
\begin{thm}[Turyn 1965 \cite{turyn}]
\label{thm:turyn}
Let $G$ be a group of order $2^{2d+2}$ containing a normal subgroup $K$ of order less than $2^d$ such that $G/K$ is cyclic. Then $G \not \in {\cal H}$.
\end{thm}

\begin{cor}[Turyn exponent bound]\label{cor:tb}
Suppose $G \in {\cal H}$ is an abelian group of order $2^{2d+2}$. Then $G$ has exponent at most $2^{d+2}$.
\end{cor}
\noindent
Dillon later proved a second nonexistence result for Hadamard $2$-groups.

\begin{thm}[Dillon 1985 \cite{dillon}]
\label{thm:dillon}
Let $G$ be a group of order $2^{2d+2}$ containing a normal subgroup $K$ of order less than $2^d$ such that $G/K$ is dihedral. Then $G \not \in {\cal H}$.
\end{thm}
\noindent
In the ensuing 35 years since the publication of \cite{dillon}, no further nonexistence results for Hadamard $2$-groups have been found.
In this paper we shall present constructive results that identify new Hadamard $2$-groups. In preparation, we introduce some further conventions that will be used throughout.

Let
$$
E_r := C_2^r = \langle x_1, x_2, \dots, x_r \rangle
$$
be the elementary abelian group of order $2^r$.
The group $E_r$ is isomorphic to the additive group of the vector space $U_r := \GF(2)^r$
comprising all binary $r$-tuples $a = (a_1, a_2, \dots, a_r)$, and an explicit isomorphism is given by
$$
a = (a_1, a_2, \dots,  a_r) \mapsto x^a = x_1^{a_1}x_2^{a_2} \cdots x_r^{a_r}. 
$$
\noindent
The \emph{characters} of $E_r$ are the homomorphisms from $E_r$ into the multiplicative group $\{1, -1\}$
given by
$$
\chi_u : x^a \mapsto (-1)^{u \cdot a} \quad \mbox{for all $a \in U_r$} 
$$
as $u$ ranges over~$U_r$.

We consider integer-valued functions on $G$ to be interchangeable with elements of~$\Z G$: 
we identify an integer-valued function $F$ on $G$ with the element $\sum_{g \in G} F(g) g$ of the group ring $\Z G$, and conversely we identify a group ring element $\sum_{g \in G} F_g g$ with the function $F$ on $G$ given by $F(g) = F_g$.
The character $\chi_u$ of $E_r$ may then be written in the group ring $\Z E_r$ as
\begin{align}
\chi_u     &= \sum_{a \in U_r} \chi_u(x^a) x^a \nonumber \\ 
           &= \sum_{a \in U_r} (-1)^{u \cdot a} x^a \nonumber \\
           &= \sum_{a \in U_r} \prod_{i=1}^r (-1)^{u_i a_i}x_i^{a_i} \nonumber \\
           &= \prod_{i=1}^r \sum_{a_i = 0}^1 (-1)^{u_i a_i}x_i^{a_i} \nonumber \\ 
           &= \prod_{i=1}^r \big (1 + (-1)^{u_i}x_i \big). \label{eqn:chiu}
\end{align}
This is consistent with the common notation $\chi_0$ for the principal character, which takes the value 1 at every group element; we identify this function in $\Z E_r$ with the group ring element $\sum_{e \in E_r} e$, or simply~$E_r$.  
For each nonzero $u \in U_r$, the complement of the subset of $E_r$ associated with the $\{\pm 1\}$-valued function $\chi_u$ is a subgroup of~$E_r$ of index~$2$, and as $u$ ranges over the nonzero values of~$U_r$ we obtain all $2^r-1$ subgroups of~$E_r$ of index~$2$ in this way.

\begin{example}
\label{ex:r=2}
Let $E_2 = C_2^2 = \langle x, y \rangle$.
The four characters of $E_2$ are the functions $\chi_u$ as $u$ ranges over $U_2$ = $\{(0, 0), (0, 1),  (1, 0), (1, 1)\}$. Expressed in the group ring $\Z E_2$, these functions are
\begin{align*}
\chi_{00} &= 1 + x + y + xy   = (1 + x)(1 + y), \\
\chi_{01} &= 1 + x - y - xy   = (1 + x)(1 - y), \\
\chi_{10} &= 1 - x + y - xy   = (1 - x)(1 + y), \\
\chi_{11} &= 1 - x - y + xy   = (1 - x)(1 - y),
\end{align*}
(where we abbreviate $\chi_{(0,1)}$, for example, as $\chi_{01}$).

The subgroups of $E_2$ corresponding to $\chi_{01}$, $\chi_{10}$, $\chi_{11}$ are $\{1,x\}$, $\{1,y\}$, $\{1,xy\}$, respectively.
\end{example}
The group ring interpretation of the characters of $E_2$ shown in \cref{ex:r=2} illustrates the following fundamental properties, which underlie our new constructions of difference sets. These properties can all be derived directly from~\eqref{eqn:chiu}, noting that $\chi_v^{(-1)} = \chi_v$ for all $v \in U_r$.

\begin{prop} 
\label{prop:orthogonality}
Let $\{\chi_u: u \in U_r\}$ be the set of characters of~$E_r$.
Then for all $u, v \in U_r$, in the group ring $\Z E_r$ we have:
\begin{enumerate}[$(i)$]

\item  $\chi_u \chi_v^{(-1)} = 
 \begin{cases} 	2^r \chi_u 	& \mbox{if $u=v$}, \\
		0		& \mbox{if $u \ne v$}
 \end{cases}$

\item  $\displaystyle{\sum_{u \in U_r}} \chi_u = 2^r$

\item  $\displaystyle{\sum_{e \in E_r}} \chi_u(e) = 
 \begin{cases} 	2^r 		& \mbox{if $u=0$}, \\
		0		& \mbox{if $u \ne 0$}.
 \end{cases}$

\end{enumerate}
\end{prop}

\noindent
Since all characters of $E_r$ are $\{\pm 1\}$-valued, \cref{prop:orthogonality}~$(iii)$ implies that every nonprincipal character on $E_r$ takes the values $1$ and $-1$ equally often.

McFarland gave the following difference set construction based on hyperplanes of a vector space, which produces examples in $2$-groups. We prove the construction by interpreting the hyperplanes in terms of characters.

\begin{thm}[McFarland hyperplane construction 1973 \cite{mcfarland}]
\label{thm:mcfarland}
Let $J$ be a group of order $2^{d+1}$. Then $J \times E_{d+1} \in {\cal H}$.
\end{thm}

\begin{proof} (Dillon \cite{dillon-galway}).
Let $\{\chi_u: u \in U_{d+1}\}$ be the set of characters of $E_{d+1}$.
Label the elements of $J$ arbitrarily as $J = \{ g_u: u \in  U_{d+1}\}$, and let $G = J \times E_{d+1}$.
We see from \cref{prop:orthogonality}~$(i)$ and~$(ii)$ that, in the group ring $\Z G$, the $\{\pm 1\}$-valued function 
\begin{equation}
\label{eqn:Dast}
D =  \sum_{u \in U_{d+1}} g_u \chi_u
\end{equation}
on $G$ satisfies 
\begin{align}
D D^{(-1)} 
& = \sum_{u,v \in U_{d+1}} g_u \chi_u \chi_v^{(-1)} g_v^{-1} \nonumber \\
& = 2^{d+1}\sum_{u \in U_{d+1}} g_u \chi_u  g_u^{-1} \label{eqn:step1} \\
& = 2^{d+1} \sum_{u \in U_{d+1}} \chi_u \label{eqn:step2} \\
& = 2^{d+1} \cdot 2^{d+1} = |G|. \nonumber
\end{align}
Therefore $D$ corresponds to a Hadamard difference set in~$G$.
 \end{proof}

We shall show how the proof of \cref{thm:mcfarland} can be adapted so that the result still holds when $E_{d+1}$ is a normal subgroup of index $2^{d+1}$ of a group $G$, but not necessarily a direct factor. The key consideration is how to obtain \eqref{eqn:step2} from~\eqref{eqn:step1}. 
The following combinatorial result allows us to do so, by showing that there is a choice for coset representatives $g_u$ of $E_{d+1}$ in $G$ satisfying $\{g_u \chi_u g_u^{-1}: u \in U_{d+1}\} = \{\chi_u: u \in U_{d+1}\}$.
Note that a group $H$ \emph{acts as a group of permutations} on a set $S$ if there is a homomorphism~$\phi$ (called the \emph{action} of $H$ on~$S$) from $H$ to the group of permutations of~$S$.

\begin{thm}[Drisko 1998 {\cite[Corollary~5]{drisko}}]
\label{thm:drisko} 
Let $p$ be a prime and let $H$ be a finite $p$-group. 
Suppose that $H$ acts as a group of permutations on a set $S$ of size $|H|$ according to the action~$\phi$, and that $S$ contains an element that is fixed under~$\phi$.
Then there is a bijection $\theta$ from $S$ to $H$ satisfying
\[
\big\{\phi\big(\theta(s)\big)(s) : s \in S\big\} = S.
\]
\end{thm}
\noindent
The bijection $\theta$ in \cref{thm:drisko} selects an element $\theta(s)$ of the group $H$ for each $s \in S$, so that the resulting set of actions of $\theta(s)$ on $s$ is a permutation of the set~$S$. We now explain how this result can be used to extend \cref{thm:mcfarland} as desired, proving a conjecture due to Dillon~\cite{dillon-MH-conf}.

\begin{cor}[Drisko 1998 {\cite[Corollary~9]{drisko}}]
\label{cor:drisko}
Let $G$ be a group of order $2^{2d+2}$ containing a normal subgroup $E \cong C_2^{d+1}$. Then $G \in {\cal H}$.
\end{cor}
\begin{proof}
Let $\widehat{E} = \{\chi_u : u \in U_{d+1}\}$ be the set of characters of $E \cong C_2^{d+1}$.
We wish to apply \cref{thm:drisko} with $S = \widehat{E}$ and $H = G/E$.
Since $E$ is normal in $G$, and the complements of the subsets of $E$ associated with the characters $\chi_u$ for nonzero~$u$ are exactly the subgroups of $E$ of index~$2$, we have
\[
g \chi_u g^{-1} \in \widehat{E} \quad \mbox{for all $g \in G$ and $\chi_u \in \widehat{E}$}.
\]
Therefore $G/E$ acts on $\widehat{E}$ as a group of permutations under the conjugation action
\[
\phi(gE)(\chi_u) = g \chi_u g^{-1} \quad \mbox{for all $gE \in G/E$ and $\chi_u \in \widehat{E}$},
\]
and the element $\chi_0 = E$ of $\widehat{E}$ is fixed under~$\phi$.
\cref{thm:drisko} then shows that there is a bijection $\theta$ from $\widehat{E}$ to $G/E$ satisfying
\begin{equation}
\label{eqn:theta}
\big\{\phi\big(\theta(\chi_u)\big)(\chi_u) : \chi_u \in \widehat{E}\big\} = \widehat{E}.
\end{equation}
Writing $\theta(\chi_u) = g_u E$ for each $u \in U_{d+1}$, this gives a set $\{g_u: u \in U_{d+1}\}$ of coset representatives for $E$ in $G$ satisfying 
\begin{equation}
\label{eqn:permute-chiu}
\{g_u \chi_u g_u^{-1} : u \in U_{d+1}\} = \{\chi_u: u \in U_{d+1}\}.
\end{equation}
Use the coset representatives $g_u$ to define $D$ as in~\eqref{eqn:Dast}.
The proof of \cref{thm:mcfarland} now carries through unchanged, using \eqref{eqn:permute-chiu} to obtain \eqref{eqn:step2} from~\eqref{eqn:step1}.
\end{proof}

We next illustrate the construction described in \cref{cor:drisko}, for a specific group of order~16.
\begin{example}
Let $G$ be the order $16$ modular group $C_8 \rtimes_{5} C_2 
 = \langle x,y : x^8=y^2=1, \, yxy^{-1} = x^5 \rangle$, and set $X = x^4$ and $Y=y$. 
Let $E = \langle X,Y \rangle \cong C_2^2$, which is normal but not central in $G$, and let $\widehat{E} = \{\chi_u: u \in U_2\}$ be the set of characters of $E$:
\[
\chi_{00} = (1+x^4)(1+y), \,\,
\chi_{01} = (1+x^4)(1-y), \,\,
\chi_{10} = (1-x^4)(1+y), \,\,
\chi_{11} = (1-x^4)(1-y).
\]
The center of $G$ is~$\langle x^2 \rangle$.

The group $G/E = \{E, xE, x^2E, x^3E\}$ acts on $\widehat{E}$ as a group of permutations under the conjugation action~$\phi$, under which $E$ and $x^2E$ map to the identity permutation on~$\widehat{E}$, and $xE$ and $x^3E$ map to the permutation of $\widehat{E}$ that fixes $\chi_{00}$ and $\chi_{01}$ but swaps $\chi_{10}$ and~$\chi_{11}$.

A bijection $\theta$ from $\widehat{E}$ to $G/E$ satisfying \eqref{eqn:theta} is 
\[
\theta(\chi_{00}) = E, \quad
\theta(\chi_{01}) = x^2E, \quad
\theta(\chi_{10}) = xE, \quad
\theta(\chi_{11}) = x^3E,
\]
and therefore
\[
D = \chi_{00} + x^2\chi_{01} + x\chi_{10} + x^3\chi_{11}
\]
is a difference set in~$G$.
\end{example}

The Turyn exponent bound of \cref{cor:tb} gives a necessary condition for an abelian $2$-group to belong to~${\cal H}$. A series of papers, including \cite{davis} and \cite{dillon2groups}, gave constructions in pursuit of a sufficient condition. Kraemer \cite{kraemer} eventually showed that the necessary condition is also sufficient. This result was proved again by Jedwab \cite{jedwab} using the alternative viewpoint of a perfect binary array: a matrix representation of the $\{\pm 1\}$-valued characteristic function of a Hadamard difference set in an abelian group.

\begin{thm}[Kraemer \cite {kraemer}]
\label{thm:kraemer}
Let $G$ be an abelian group of order $2^{2d+2}$. Then $G \in {\cal H}$ if and only if $G$ has exponent at most $2^{d+2}$. 
\end{thm}

We next give an instructive example of a Hadamard difference set in an abelian $2$-group, which illustrates a fundamental insight on which this paper is based.
The group ring elements $A_u$ in \cref{ex:C82} are presented for now without explanation of their origin, but will be revisited in \cref{ex:C82revisit}.
Group ring elements $A, B$ are {\em orthogonal} if $AB^{(-1)} = 0$.

\begin{example}
\label{ex:C82}
Let 
$G = C_8^2 = \langle x, y \rangle$, and set $X = x^2$ and $Y = y^2$. 
Let $K=\langle X, Y \rangle \cong C_4^2$ and
$E_2 =\langle X^2, Y^2 \rangle \cong C_2^2$, and let
$\{\chi_u: u \in U_2\}$ be the set of characters of~$E_2$.
Define four group ring elements in $\Z K$ by
\begin{equation}
\label{eqn:Auturyn}
A_{00} = A_{01} = A_{10} = 1 + X + Y - X Y \quad \mbox{and} \quad
A_{11} = 1 + X + Y + X Y.
\end{equation}
Direct calculation shows that the $A_u$ satisfy the condition
\begin{equation}
\label{eqn:Aupropturyn}
A_u \chi_u A_u^{(-1)} = 4 \chi_u \quad \mbox{for all $u \in U_2$}.
\end{equation}
Now in $\Z K$ let 
\begin{align*}
B_{00} &= A_{00} \chi_{00} = (1 + X + Y - X Y) (1+X^2)(1+Y^2), \\[0.5ex]
B_{01} &= A_{01} \chi_{01} = (1 + X + Y - X Y) (1+X^2)(1-Y^2), \\[0.5ex]
B_{10} &= A_{10} \chi_{10} = (1 + X + Y - X Y) (1-X^2)(1+Y^2), \\[0.5ex]
B_{11} &= A_{11} \chi_{11} = (1 + X + Y + X Y) (1-X^2)(1-Y^2). 
\end{align*}
Then from \cref{prop:orthogonality}~(i) and \eqref{eqn:Aupropturyn}, the $B_u = A_u \chi_u$ have the property, for all $u, v \in U_2$, that
\begin{equation}
\label{eqn:Bupropturyn}
B_u B_v^{(-1)} = 
  \begin{cases} 16 \chi_u & \mbox{if $u=v$}, \\
		0	  & \mbox{if $u \ne v$},
  \end{cases} 
\end{equation}
and in particular the $B_u$ are pairwise orthogonal.
It follows that the $\{\pm 1\}$-valued function on $G$ given by
\[
D = B_{00} + y B_{01} + x B_{10} + xy B_{11}
\]
satisfies 
\begin{align*}
D D^{(-1)} 
 &= 16(\chi_{00} + \chi_{01} + \chi_{10} + \chi_{11}) \\ 
 &= 64 
\end{align*}
by \cref{prop:orthogonality}~(ii), and so $D$ corresponds to a Hadamard difference set in~$G$.
\end{example}
We now show how the condition \eqref{eqn:Aupropturyn} satisfied by the group ring elements $A_u$ in \cref{ex:C82} can be used to construct difference sets in groups of order~64 other than~$C_8^2$.

\begin{prop}
\label{prop:dillon64}
Let $G$ be a group of order $64$ containing a normal subgroup $K \cong C_4^2$. Then $G \in {\cal H}$.
\end{prop}
\begin{proof}
Let $K = \langle X, Y \rangle \cong C_4^2$.
Let $E_2 = \langle X^2, Y^2 \rangle$ be the unique subgroup of $K$ isomorphic to $C_2^2$, and let $\widehat{E_2} = \{\chi_u : u \in U_2\}$ be the set of characters of~$E_2$.
Define four group ring elements in $\Z K$ as in \eqref{eqn:Auturyn},
and for each $u \in U_2$ let $B_u$ be the $\{\pm 1\}$-valued function $A_u \chi_u$ on~$K$.
The $A_u$ satisfy \eqref{eqn:Aupropturyn}, and therefore the $B_u$ have the pairwise orthogonality property~\eqref{eqn:Bupropturyn} for all $u,v \in U_2$.

Now $E_2$ is the unique subgroup of $K$ isomorphic to $C_2^2$, and $K$ is normal in $G$, so $E_2$ is normal in~$G$.
Therefore $G/K$ acts on $\widehat{E_2}$ as a group of permutations under the conjugation action
\[
\phi(gK)(\chi_u) = g \chi_u g^{-1} \quad \mbox{for all $gK \in G/K$ and $\chi_u \in \widehat{E_2}$},
\]
and $\chi_0 = E_2$ is fixed under~$\phi$.
We may therefore apply \cref{thm:drisko} with $S = \widehat{E_2}$ and $H = G/K$ to show that
there is a set $\{g_u : u \in U_{2}\}$ of coset representatives for $K$ in $G$ satisfying
\begin{equation}
\label{eqn:permute-chiu2}
\{g_u \chi_u g_u^{-1} : u \in U_2\} = \{\chi_u: u \in U_2\}.
\end{equation}
Let $D$ be the $\{\pm 1\}$-valued function on $G$ defined by
\[
D = \sum_{u \in U_2} g_u B_u \mbox{ in $\Z G$}. 
\]
We calculate
\begin{align*}
DD^{(-1)}
  &= \sum_{u,v \in U_2} g_u B_u B_v^{(-1)}g_v^{-1}\\
  &= 16 \sum_{u \in U_2} g_u \chi_u g_u^{-1}
\end{align*}
by \eqref{eqn:Bupropturyn}, and then from \eqref{eqn:permute-chiu2} and \cref{prop:orthogonality}~$(ii)$ we have
\[
DD^{(-1)} = 16 \sum_{u \in U_2} \chi_u = 64.
\]
Therefore $D$ corresponds to a Hadamard difference set in~$G$.
\end{proof}

We use the proof of \cref{prop:dillon64} as a model for establishing our principal result, stated below as \cref{thm:main}. The key idea is to determine group ring elements $A_u$ satisfying a condition analogous to~\eqref{eqn:Aupropturyn}, which ensures that the associated group ring elements $B_u = A_u \chi_u$ have an orthogonality property analogous to~\eqref{eqn:Bupropturyn}. Application of \cref{thm:drisko} then allows us to construct a group ring element $D$ corresponding to a Hadamard difference set.
By taking $r=2$ in \cref{thm:main} and restricting the group $G$ to be abelian, and combining with the Turyn exponent bound of \cref{cor:tb}, we recover Kraemer's \cref{thm:kraemer}.

\begin{thm}[Main Result]
\label{thm:main}
Let $d$ and $r$ be integers satisfying $d \ge 1$ and $2 \le r \le d+1$. Let $G$ be a group of order $2^{2d+2}$ containing a normal abelian subgroup of index $2^r$, rank~$r$, and exponent at most~$2^{d-r+2}$. Then $G \in {\cal H}$. 
\end{thm}

We remark that this paper develops several concepts previously used to construct difference sets. 
In particular, the constructed group ring elements $B_u$ can be interpreted as covering extended building sets, as introduced by Davis and Jedwab~\cite{unifying} in 1997 (see the discussion at the end of \cref{sec:sig}). The novelty here is that imposing the additional structure $B_u = A_u \chi_u$ allows us to handle dramatically more nonabelian groups than before, as illustrated in the proof of \cref{prop:dillon64}. 
Likewise, \cref{prop:dillon64} itself was previously established by Dillon \cite{dillon-MH-conf, dillon-galway} by decomposing a difference set in $C_8^2$ into four orthogonal group ring elements $B_u$ as in \cref{ex:C82}.
However, the generalization of \cref{prop:dillon64} to \cref{thm:main} relies crucially on recognizing the additional structure $B_u = A_u \chi_u$ of these group ring elements, whose importance was not previously apparent.

The third column of \cref{tab:2-groups} below shows the number of groups of order 16, 64, and 256 which are 
possible members of~${\cal H}$, after taking into account those that are excluded by the necessary conditions of Theorems~\ref{thm:turyn} and~\ref{thm:dillon}.
We now summarize the theoretical and computational efforts of many researchers over several decades to determine whether these conditions are also sufficient for groups of these orders, with reference to results to be presented in \cref{sec:further}.

In the 1970s, Whitehead~\cite{whitehead} and Kibler~\cite{kibler} independently showed by construction that each of the 12 non-excluded groups of order~16 belongs to~${\cal H}$.
We can recover this result by applying \cref{thm:main} to account for the 10 groups containing a normal subgroup isomorphic to $C_2^2$, and then using \cref{prop:sigsetQ} to handle the remaining 2 groups.

In 1990, a collaborative effort led by Dillon showed by a combination of construction and computer search that each of the 259 non-excluded groups of order~64 belongs to~${\cal H}$;
Liebler and Smith~\cite{smithliebler} resolved the status of the final group at the conclusion of a sabbatical visit to Dillon by Smith.
Using the GAP software package~\cite{gap}, we can streamline this effort by applying in sequence the following construction methods: 
\cref{thm:main} to account for the 237 groups containing a normal subgroup isomorphic to $C_2^3$ or~$C_4^2$; 
the product construction of \cref{prop:productPTAgivesDS} to account for 17 further groups; 
the transfer methods of \cref{subsec:transfer} to account for 4 further groups; 
and the modified signature set method of \cref{subsec:final} to account for the final group.

In 2011, Dillon initiated a further collaborative effort to determine which of the 56,049 non-excluded groups of order 256 belong to~${\cal H}$. 
Major contributions were made by Applebaum~\cite{applebaum}, and the status of the final group was resolved by Yolland~\cite{yolland} in 2016. 
Using GAP and again streamlining, we announce that all 56,049 non-excluded groups of order 256 belong to~${\cal H}$, and this can be demonstrated by applying in sequence the following construction methods:
\cref{thm:main} to account for the 54,633 groups containing a normal subgroup isomorphic to $C_2^4$ or $C_4^2 \times C_2$ or $C_8^2$;
the product construction of \cref{prop:productPTAgivesDS} to account for 1,358 further groups;
the transfer methods of \cref{subsec:transfer} to account for 57 further groups;
and the modified signature set method of \cref{subsec:final} to account for the final group.

These theoretical and computational results are summarized  in \cref{tab:2-groups}.

\begin{table}[h]
\begin{center}
\begin{tabular}{|c|c|c|c|c|c|c|}
												     \hline
Group 	& Total \# 	& \# not excluded	
					& \multicolumn{4}{c|}{\# in ${\cal H}$ by} 		  \\ \cline{4-7}
order	& groups	& by Theorems 	& Theorem  	& Sections 	& Section 	& Section \\ 
	&		& \ref{thm:turyn}, \ref{thm:dillon}	
					& \ref{thm:main}
							& \ref{subsec:nonabelian-sigset}--\ref{subsec:productPTA}
									& \ref{subsec:transfer}
											& \ref{subsec:final}
												\\ \hline
16	& 14		& 12		& 10		& 2		&  		&  	\\
64	& 267		& 259		& 237		& 17 		& 4		& 1	\\
256	& 56,092	& 56,049	& 54,633	& 1,358		& 57	& 1	\\ \hline
\end{tabular}
\end{center}
\caption{Membership in ${\cal H}$ of $2$-groups of order 16, 64, and 256.
Figures in column 5 onwards are for groups not previously counted in column 4 onwards.}
\label{tab:2-groups}
\end{table}

The results displayed in \cref{tab:2-groups} naturally prompt the following question (about whose answer the authors of this paper have different opinions).
\begin{question}
\label{quest:the-big-one}
Are the necessary conditions of Theorems~$\ref{thm:turyn}$ and~$\ref{thm:dillon}$ for the existence of a difference set in a $2$-group also sufficient? That is, does every group $G$ of order $2^{2d+2}$, not containing a normal subgroup $K$ of order less than $2^d$ such that $G/K$ is cyclic or dihedral, belong to ${\cal H}$?
\end{question}

We have seen that the answer to \cref{quest:the-big-one} is ``yes'' for $d = 0, 1, 2, 3$ (noting for $d=0$ that both groups of order 4 contain a trivial difference set). It seems that resolution of this question for larger $d$ must depend only on theoretical methods: currently there is not even a database of the 49,487,367,289 groups of order 1024 \cite{besche-eick-obrien,burrell}, and the authors do not know how to estimate the proportion of the non-excluded groups of order $2^{2d+2}$ that are accounted for by \cref{thm:main} as $d$ grows large.

The rest of this paper is organized in the following way.
In \cref{sec:sig}, we identify the ``signature set'' property underlying the construction of \cref{prop:dillon64}.
In \cref{sec:abelian}, we prove our principal result of \cref{thm:main} by restricting attention to signature sets on abelian $2$-groups. 
In \cref{sec:further}, we describe the various other construction methods used to complete the determination of the groups of order 64 and 256 belonging to ${\cal H}$, involving signature sets on nonabelian groups, products of perfect ternary arrays, transfer methods, and a modification of signature sets.
In \cref{sec:verification}, we provide implementation details of the construction methods for groups of order 256 and describe how to quickly verify on a desktop computer that all 56,049 non-excluded groups of this order belong to ${\cal H}$.
In \cref{sec:future}, we propose some directions for future research.

\section{Signature Sets}
\label{sec:sig}
In this section, we identify the structure underlying \cref{prop:dillon64} and set out a framework for proving our principal result, \cref{thm:main}.

\begin{defn}
\label{defn:sig}
Let $K$ be a group containing a normal subgroup $E \cong C_2^r$, and let $\{\chi_u: u\in U_r\}$ be the set of characters of~$E$. 
A \emph{signature block on $K$ with respect to $\chi_u$} is a $\{\pm 1\}$-valued function $A_u$ on a set of coset representatives for $E$ in $K$ that satisfies
\[
A_u \chi_u A_u^{(-1)} = \tfrac{|K|}{2^r} \chi_u \quad \mbox{in $\Z K$}. 
\]
A \emph{signature set on $K$ with respect to $E$} is a multiset $\{A_u: u \in U_r\}$, where each $A_u$ is a signature block on~$K$ with respect to~$\chi_u$.
\end{defn}
\noindent
Note that a trivial signature set on $C_2^r$ with respect to itself is given by
\[
A_u = 1 \quad \mbox{for each $u \in U_r$}.
\]

We state two immediate consequences of \cref{defn:sig}.
\begin{lemma}
\label{lem:sig}
Let $K$ be a group containing a normal subgroup $E \cong C_2^r$, and suppose $\{A_u: u \in U_r\}$ is a signature set on $K$ with respect to~$E$. 
Let $\widehat{E} = \{\chi_u: u \in U_r\}$ be the set of characters of $E$, and let $B_u = A_u \chi_u$ for each $u \in U_r$. Then:
\begin{enumerate}[$(i)$]
\item
for each $u \in U_r$, the function $B_u$ is $\{\pm 1\}$-valued on~$K$.

\item
for all $u, v \in U_r$, in $\Z K$ we have
\[
B_u B_v^{(-1)} 
 = \begin{cases} |K| \chi_u 	& \mbox{if $u=v$}, \\
		  0	     	& \mbox{if $u \ne v$}
    \end{cases} 
\]
(and so in particular the $B_u$ are pairwise orthogonal).
\end{enumerate}
\end{lemma}

\begin{proof}
\begin{enumerate}[$(i)$]
\item
Each $A_u$ is a $\{\pm 1\}$-valued function on a set of coset representatives for $E$ in $K$, and each $\chi_u$ is a $\{\pm 1\}$-valued function on~$E$. Therefore each $B_u = A_u \chi_u$ is a $\{\pm 1\}$-valued function on~$K$.

\item
For all $u, v \in U_r$, in $\Z K$ we have
\begin{align*}
B_u B_v^{(-1)} 
 &= A_u \chi_u \chi_v^{(-1)} A_v^{(-1)} \nonumber \\
 &= \begin{cases} 2^r A_u \chi_u A_u^{(-1)} 	& \mbox{if $u=v$}, \\
		  0	     			& \mbox{if $u \ne v$}
    \end{cases} 
\end{align*}
by \cref{prop:orthogonality}~$(i)$. Since the $A_u$ form a signature set on $K$ with respect to $E$, this gives
\[
B_u B_v^{(-1)} = 
    \begin{cases} |K| \chi_u 	& \mbox{if $u=v$}, \\
		  0		& \mbox{if $u \ne v$}.
    \end{cases} 
\]

\end{enumerate}
\end{proof}
\noindent

The proof of the following theorem is modelled on that of \cref{prop:dillon64}.
We remark that $K$ need not be a $2$-group and need not be abelian.

\begin{thm} 
\label{thm:prehadamard}
Let $G$ be a group containing a normal subgroup $E \cong C_2^r$,
and suppose $K$ is a normal subgroup of $G$ of index~$2^r$ containing~$E$.
Suppose there exists a signature set on $K$ with respect to~$E$.
Then $G \in {\cal H}$.
\end{thm}

\begin{proof}
Let $\widehat{E} = \{\chi_u: u \in U_r\}$ be the set of characters of $E$.
We shall apply \cref{thm:drisko} with $S = \widehat{E}$ and $H = G/K$.
Since $E$ is normal in $G$, and the complements of the subsets of $E$ associated with the characters $\chi_u$ for nonzero~$u$ are exactly the subgroups of $E$ of index~$2$,
\[
g \chi_u g^{-1} \in \widehat{E} \quad \mbox{for all $g \in G$ and $\chi_u \in \widehat{E}$}.
\]
Therefore $G/K$ acts on $\widehat{E}$ as a group of permutations under the conjugation action
\[
\phi(gK)(\chi_u) = g \chi_u g^{-1} \quad \mbox{for all $gK \in G/K$ and $\chi_u \in \widehat{E}$},
\]
and the element $\chi_0 = E$ of $\widehat{E}$ is fixed under~$\phi$.
Apply \cref{thm:drisko} to show that there is a set $\{g_u: u \in U_r\}$ of coset representatives for $K$ in $G$ satisfying 
\begin{equation}
\label{eqn:permute-chiu3}
\{g_u \chi_u g_u^{-1} : u \in U_r\} = \{\chi_u: u \in U_r\}.
\end{equation}

By assumption, there is a signature set $\{A_u: u \in U_r\}$ on $K$ with respect to~$E$.
Let $B_u = A_u \chi_u$ for each $u \in U_r$, and use the coset representatives~$g_u$ to define
\begin{equation}
\label{eqn:Dast2}
D = \sum_{u \in U_r} g_u B_u \quad \mbox{in $\Z G$},
\end{equation}
which is a $\{\pm 1\}$-valued function on $G$ by \cref{lem:sig}~$(i)$.
We calculate in $\Z G$ that
\begin{align*}
DD^{(-1)} 
  &= \sum_{u,v \in U_r} g_u B_u B_v^{(-1)}g_v^{-1}\\
  &= |K| \sum_{u \in U_r} g_u \chi_u g_u^{-1}
\end{align*}
by \cref{lem:sig}~$(ii)$.
Then from \eqref{eqn:permute-chiu3} and \cref{prop:orthogonality}~$(ii)$ we have
\[
DD^{(-1)} = |K| \sum_{u \in U_r} \chi_u = 2^r |K| = |G|.
\]
Therefore $D$ corresponds to a Hadamard difference set in~$G$.
\end{proof}

The motivating examples of \cref{sec:intro} both occur as special cases of \cref{thm:prehadamard}.
\cref{cor:drisko} arises by taking $|G| = 2^{2d+2}$ and $r=d+1$, with $E = K \cong C_2^{d+1}$ normal in~$G$, and using a trivial signature set on $K$ with respect to itself.
\cref{prop:dillon64} arises by taking $|G|=64$ and $r=2$, with $K = \langle X, Y \rangle \cong C_4^2$ normal in $G$ and $E = \langle X^2, Y^2 \rangle$ (the unique subgroup of $K$ isomorphic to $C_2^2$), and 
using the nontrivial signature set $\{A_{ij} : (i,j) \in U_2\}$ on $K$ with respect to~$E$ specified in~\eqref{eqn:Auturyn}.

\cref{thm:prehadamard} establishes the existence of a difference set in $G$ by reference to \cref{thm:drisko}, whose proof as given in \cite{drisko} is not constructive. To construct such a difference set explicitly, one must therefore determine suitable coset representatives for the normal subgroup $K$ in $G$ satisfying \eqref{eqn:permute-chiu3}. This determination currently requires a computer search that can be computationally expensive, particularly for groups of order 256 (see \cref{sec:verification}).

We point out a connection to the study of bent functions (see \cite{carmes} for a survey), which are equivalent to Hadamard difference sets in elementary abelian $2$-groups. 
Take $G = E_{d+1}^2$ and $E = K = E_{d+1}$ in \cref{thm:prehadamard}, and 
let $\{A_u: u \in U_r\}$ be a trivial signature set on $K$ with respect to~$E$ for which each $A_u$ is chosen arbitrarily in~$\{\pm 1\}$.
In this case, the choice of coset representatives $\{g_u: u \in U_{d+1}\}$ for $K$ in $G$ used to construct the difference set $D$ in the proof of \cref{thm:prehadamard} is arbitrary.
Let $a$ be the Boolean function on~$U_{d+1}$ defined by
\[
A_u = (-1)^{a(u)} \quad \mbox{for each $u \in U_{d+1}$}.
\]
Then the $\{0,1\}$-valued characteristic function of $D$ is the Maiorana-McFarland bent function $f(u,v) = \pi(u) \cdot v + a(u)$, where $\pi$ is an arbitrary permutation of~$U_{d+1}$.

In view of \cref{thm:prehadamard}, our objective in \cref{sec:abelian} is to construct a signature set on a large class of groups~$K$ 
(which we take to be abelian in \cref{sec:abelian}, and nonabelian in \cref{sec:further}).
In the remainder of this section, we introduce some preparatory results about signature sets.

We firstly show that a group automorphism of $K$ fixing $E$ maps a signature block on $K$ to another signature block on~$K$.
\begin{prop} 
\label{prop:Aaut}
Let $K$ be a group containing a normal subgroup $E \cong C_2^r$, and let $\sigma$ be a group automorphism of $K$ which fixes~$E$. Suppose that $A_u$ is a signature block on $K$ with respect to the character $\chi_u$ of~$E$, for some $u \in U_r$. Then $\sigma$ induces a map on $\Z K$ under which $\sigma(A_u)$ is a signature block on $K$ with respect to the character $\sigma(\chi_u)$ of~$E$.
\end{prop}
\begin{proof}
The signature block $A_u$ is $\{\pm 1\}$-valued on a set of coset representatives for $E$ in~$K$. Since the automorphism $\sigma$ fixes $E$, the images of these coset representatives under $\sigma$ are also a set of coset representatives for $E$ in~$K$ on which $\sigma(A_u)$ is $\{\pm 1\}$-valued. Furthermore
\begin{align*}
\sigma(A_u) \sigma(\chi_u) \sigma(A_u)^{(-1)}
 &= \sigma (A_u \chi_u A_u^{(-1)}) \\
 &= \tfrac{|K|}{2^r} \sigma (\chi_u),
\end{align*}
so $\sigma(A_u)$ is a signature block on $K$ with respect to the character $\sigma(\chi_u)$ of~$E$.
\end{proof}

We next give a simple product construction for signature sets.
\begin{prop}
\label{prop:sig-prod}
Suppose there exists a signature set on a group $K_r$ with respect to a normal subgroup $E_r \cong C_2^r$, and there exists a signature set on a group $K_s$ with respect to a normal subgroup $E_s \cong C_2^s$. Then there exists a signature set on $K_r \times K_s$ with respect to $E_r \times E_s$.
\end{prop}
\begin{proof}
Let $\{A_u: u \in U_r\}$ be a signature set on $K_r$ with respect to $E_r$, and 
let $\{\alpha_v: v \in U_s\}$ be a signature set on $K_s$ with respect to~$E_s$.
We claim that $\{A_u \alpha_v: u \in U_r, v \in U_s\}$ is a signature set on $K_r \times K_s$ with respect to its normal subgroup~$E_r \times E_s$.

The function $A_u \alpha_v$ is $\{\pm 1\}$-valued on a set of coset representatives for $E_r \times E_s$ in $K_r \times K_s$, because
$A_u$ is $\{\pm 1\}$-valued on a set of coset representatives for $E_r$ in~$K_r$ and
$\alpha_v$ is $\{\pm 1\}$-valued on a set of coset representatives for $E_s$ in~$K_s$.

Let $\{\chi_u : u \in U_r\}$ be the set of characters of~$E_r$, and 
let $\{\psi_v : v \in U_s\}$ be the set of characters of~$E_s$.
The set of characters of $E_r \times E_s$ is $\{\chi_u \psi_v: u \in U_r, v \in U_s\}$, and for each $u \in U_r$ and $v \in U_s$ we have
\begin{align*}
(A_u \alpha_v) (\chi_u \psi_v) (A_u \alpha_v)^{(-1)}
 &= A_u \chi_u \big ( \alpha_v \psi_v \alpha_v^{(-1)} \big)  A_u^{(-1)} \\
 &= A_u \chi_u \tfrac{|K_s|}{2^s} \psi_v A_u^{(-1)} \\
 &= \big( A_u \chi_u A_u^{(-1)} \big) \tfrac{|K_s|}{2^s} \psi_v \\
 &= \tfrac{|K_r|}{2^r} \chi_u \, \tfrac{|K_s|}{2^s} \psi_v \\
 &= \tfrac{|K_r \times K_s|}{2^{r+s}} (\chi_u \psi_v).
\end{align*}
\end{proof}
\noindent
To illustrate the previously unrecognized power of the signature set approach, note that in 2013 Applebaum~\cite{applebaum} used computer search to show that 643 of the 714 groups of order 256, whose membership in ${\cal H}$ was then undetermined, belong to ${\cal H}$. Since all 643 of these groups contain a normal subgroup isomorphic to $C_4^2 \times C_2$, this result follows directly from \cref{thm:prehadamard} simply by exhibiting a signature set on $C_4^2 \times C_2$ with respect to its unique subgroup isomorphic to $C_2^3$. This can be constructed by using \cref{prop:sig-prod} to take the product of a signature set on $C_4^2$ with respect to its unique subgroup isomorphic to $C_2^2$ (see \cref{ex:C82}) with a trivial signature set on $C_2$ with respect to itself.

Finally, we derive constraints on a signature set in terms of $|K|$ and~$|E|$. We will use these constraints to show how \cref{thm:prehadamard} can be viewed as refining a construction method for difference sets introduced by Davis and Jedwab~\cite{unifying}, by interpreting a signature set on an abelian group as a special kind of covering extended building set. 

\begin{lemma}
\label{lem:-1} 
Let $K$ be a group containing a normal subgroup $E \cong C_2^r$, and suppose that 
$\{A_u : u \in U_r\}$ is a signature set on $K$ with respect to~$E$.
Let $\{\chi_u: u \in U_r\}$ be the set of characters of $E$, and let $B_u = A_u \chi_u$ for each $u \in U_r$. 
Then the number of times the $\{\pm 1\}$-valued function $B_u$ on $K$ takes the value $-1$ is 
\[
\begin{cases}
 \frac{1}{2}|K| 			& \mbox{if $u \ne 0$}, \\[1ex]
 \frac{1}{2}|K| \pm \sqrt{2^{r-2}|K|}	& \mbox{if $u = 0$}.
\end{cases}
\]
\end{lemma}
\begin{proof}
By \cref{lem:sig}~$(i)$, each $B_u$ is $\{\pm 1\}$-valued on~$K$.

\begin{description}
\item[Case 1: $u \ne 0$.]
By \cref{prop:orthogonality}~$(iii)$, the number of times the $\{\pm 1\}$-valued function $\chi_u$ on $E$ takes the value $-1$ is $\frac{1}{2}|E|$. Since $A_u$ is a $\{\pm 1\}$-valued function on a set of coset representatives for $E$ in $K$, the number of times $B_u = A_u \chi_u$ takes the value $-1$ is $\frac{1}{2}|E| |K:E| = \frac{1}{2}|K|$.

\item[Case 2: $u = 0$.]
Let $c \in \{0, 1, \dots, |K|\}$ be the number of times that $B_0$ takes the value~$-1$,
and let $J$ be a group of order~$2^{r}$.
By \cref{thm:prehadamard}, the group $G = J \times K$ contains a Hadamard difference set $D$ whose corresponding $\{\pm 1\}$-valued function is defined in \eqref{eqn:Dast2} as
\begin{equation}
\label{eqn:Dast3}
D = g_0 B_0 + \sum_{u \ne 0} g_u B_u 
\end{equation}
for some choice of coset representatives $\{g_u: u \in U_r\}$ for $K$ in~$G$. 
By \eqref{eqn:had_params}, the parameters of the difference set $D$ satisfy
\[
|G| = 2^r |K| = 4N^2 \quad \mbox{and} \quad |D| = 2N^2-N
\]
for some integer~$N$, and eliminating $N$ gives
\[
|D| = 2^{r-1} |K| \pm \sqrt{2^{r-2}|K|}.
\]
But $|D|$ equals the number of times that the function $D$ takes the value~$-1$, which from \eqref{eqn:Dast3} and the result for Case~1 gives
\[
|D| = c + (2^r-1)\tfrac{1}{2}|K|.
\]
Equate the two expressions for $|D|$ to give
\[
c = \tfrac{1}{2}|K| \pm \sqrt{2^{r-2}|K|}.
\]
\end{description}
\end{proof}
\noindent
Note from \cref{ex:C82} that the number of times the function $A_u$ takes the value~$-1$ is not determined for $u \ne 0$ solely from the hypotheses of \cref{lem:-1}. However, for $u=0$ this number is determined as 
$\frac{1}{2^r}\Big( \frac{|K|}{2} \pm \sqrt{2^{r-2}|K|} \Big)$ by \cref{lem:-1} and the relation $B_0 = A_0 \chi_0$, because the $\{\pm 1\}$-valued function $\chi_0 = E$ takes the value $1$ exactly $2^r$ times.

We can now interpret \cref{thm:prehadamard} in the framework of~\cite{unifying} for the case that $K$ is abelian. Suppose $\{A_u : u \in U_r\}$ is a signature set on an abelian group $K$ with respect to $E =\langle x_1,x_2,\dots,x_r \rangle \cong C_2^r$, and let $B_u = A_u \chi_u$ for each $u \in U_r$.  
In the language of \cite{unifying}, we claim that the subsets $\big\{\frac{1}{2}(K-B_u) : u \in U_r\big\}$ of $K$ then form a $(\frac{|K|}{2}, \sqrt{2^{r-2}|K|}, 2^r, \pm)$ covering extended building set on~$K$ (satisfying the key additional constraint that $B_u = A_u \chi_u$ for each~$u$).
To prove the claim, we require firstly that 
\[
\big|\tfrac{1}{2}(K-B_u)\big| = 
\begin{cases}
 \frac{1}{2}|K| \pm \sqrt{2^{r-2}|K|}	& \mbox{for a single value of $u$}, \\[1ex]
 \frac{1}{2}|K| 			& \mbox{for all other values of $u$}.
\end{cases}
\]
This is given by \cref{lem:-1}, because $\big|\tfrac{1}{2}(K-B_u)\big|$ is the number of times that the $\{\pm 1\}$-valued function $B_u$ takes the value~$-1$.
To complete the proof of the claim, we also require that, for each nonprincipal character $\psi$ of the abelian group $K$ (namely a nontrivial homomorphism from $K$ to the complex roots of unity),
\[
\Big|\psi\big(\tfrac{1}{2}(K-B_u)\big)\Big| = 
  \begin{cases} 	
	\sqrt{2^{r-2}|K|} 	& \mbox{for a single value of $u$ that depends on $\psi$}, \\
	0 			& \mbox{for all other values of $u$}.
  \end{cases}
\]
This is given by applying $\psi$ to the case $u=v$ of \cref{lem:sig}~$(ii)$ to obtain
$|\psi(B_u)|^2 = |K| \psi(\chi_u)$, and noting that $\psi$ maps each $x_i$ to $\{1,-1\}$ so that from \eqref{eqn:chiu} we have
\[
\psi(\chi_u) = 
  \begin{cases} 	
	2^r 			& \mbox{for a single value of $u$ that depends on $\psi$}, \\
	0 			& \mbox{for all other values of $u$}.
  \end{cases}
\]

\section{Proof of Main Result}
\label{sec:abelian}

In this section we prove our main result, \cref{thm:main}, as a corollary of \cref{thm:abelian-sigset} below.
For an abelian $2$-group $K$ of rank~$r$, we shall abbreviate 
``a signature set on $K$ with respect to its unique subgroup isomorphic to $C_2^r$\,'' as 
``a signature set on~$K$''. 

\begin{thm} 
\label{thm:abelian-sigset}
Let $d$ and $r$ be integers satisfying $d \ge 1$ and $2 \le r \le d+1$.
Let ${\cal K}_{d,r}$ be the set of all abelian groups of order $2^{2d-r+2}$, rank~$r$, and exponent at most~$2^{d-r+2}$.
Then there exists a signature set on each $K_{d,r} \in {\cal K}_{d,r}$.
\end{thm}
\noindent
Note in \cref{thm:abelian-sigset} that if $E$ is the unique subgroup of $K_{d,r} \in {\cal K}_{d,r}$ isomorphic to $C_2^r$, then $E$ is normal in~$G$. We may therefore apply \cref{thm:prehadamard} to obtain \cref{thm:main} as a corollary of \cref{thm:abelian-sigset}.

We shall prove \cref{thm:abelian-sigset} using a recursive construction for signature sets on abelian $2$-groups. To illustrate the main ideas, we begin with a proof of the special case $r=2$. 

\begin{thm}[Rank~$2$ case of~\cref{thm:abelian-sigset}]
\label{thm:rank2}
Let $d$ be a non-negative integer. 
Then there exists a signature set on $K_d = C_{2^d}^2$. 
\end{thm}
\begin{proof} 
The proof is by induction on $d \ge 1$.
The case $d=1$ is true, because there exists a trivial signature set on $C_2^2$.

Assume all cases up to $d-1 \ge 1$ are true. Let 
$K_{d-1} = \langle X, Y \rangle$, where $X^{2^{d-1}} = Y^{2^{d-1}} = 1$.
By the inductive hypothesis, there exists a signature set $\{A_{ij} : (i,j) \in U_2\}$ on~$K_{d-1}$ with respect to $\langle X^{2^{d-2}}, Y^{2^{d-2}} \rangle$. 
By associating the group ring $\Z K_{d-1}$ with the quotient ring $\Z[X,Y]/\langle 1-X^{2^{d-1}}, 1-Y^{2^{d-1}} \rangle$, we may regard
each group ring element $A_{ij}$ as a polynomial $A_{ij}(X,Y)$ in $X$ and~$Y$,
and regard each character of $\langle X^{2^{d-2}}, Y^{2^{d-2}} \rangle$ as a polynomial
\[
\chi_{ij}(X,Y) = \big(1+(-1)^{i} X^{2^{d-2}}\big) \big(1+(-1)^{j} Y^{2^{d-2}}) \quad \mbox{for $(i, j) \in U_2$}.
\]
By assumption, in the polynomial ring $\Z[X,Y]/\langle 1-X^{2^{d-1}}, 1-Y^{2^{d-1}} \rangle$ we have
\begin{equation}
\label{eqn:AuXY}
A_{ij}(X,Y) \chi_{ij}(X,Y) A_{ij}(X,Y)^{(-1)} = 2^{2d-4} \chi_{ij}(X,Y) \quad \mbox{for each $(i,j) \in U_2$}.
\end{equation}

Let $K_d = \langle x, y \rangle$, where $x^{2^d} = y^{2^d} = 1$, and let $E = \langle x^{2^{d-1}}, y^{2^{d-1}} \rangle$.
We wish to construct a signature set $\{\alpha_{ij}: (i,j) \in U_2\}$ on $K_d$ with respect to $E$.
Define the $\alpha_{ij}$ in $\Z K_d$ in terms of the polynomials~$A_{ij}$ via
\begin{equation}
\label{eqn:rank2recursion}
\left .
\begin{array}{rl}
\alpha_{00} &= (1+x^{2^{d-2}})A_{00}(x,y^2) + y(1-x^{2^{d-2}})A_{10}(x,y^2), \\[1ex]
\alpha_{01} &= (1+x^{2^{d-2}})A_{01}(x,y^2) + y(1-x^{2^{d-2}})A_{11}(x,y^2), \\[1ex]
\alpha_{10} &= (1+y^{2^{d-2}})A_{10}(x^2,y) + x(1-y^{2^{d-2}})A_{11}(x^2,y), \\[1ex]
\alpha_{11} &= (1+x^{2^{d-2}}y^{2^{d-2}})A_{10}(x^2,xy) + x(1-x^{2^{d-2}}y^{2^{d-2}})A_{11}(x^2,xy), 
\end{array}
\right \} 
\end{equation}
and let the characters of $E$ be
\[
\psi_{ij} = \big (1+(-1)^i x^{2^{d-1}} \big) \big (1+(-1)^j y^{2^{d-1}} \big) \quad \mbox{for each $(i, j) \in U_2$}.
\]

We firstly use \cref{prop:Aaut} to show it is sufficient to prove for each $(i,j) \ne (1,1)$ that $\alpha_{ij}$ is a signature block with respect to~$\psi_{ij}$.
Let $\sigma$ be the group automorphism of $K_d$ that maps $x$ to itself and maps $y$ to~$xy$. 
Then $\sigma(\alpha_{10}) = \alpha_{11}$ by definition, and $\sigma$ fixes $E$, and
\[\sigma(\psi_{10}) = (1-x^{2^{d-1}})(1+x^{2^{d-1}}y^{2^{d-1}}) = (1-x^{2^{d-1}})(1-y^{2^{d-1}}) = \psi_{11}.
\]
Therefore if $\alpha_{10}$ is a signature block on $K_d$ with respect to $\psi_{10}$, then $\alpha_{11}$ is a signature block on $K_d$ with respect to~$\psi_{11}$ by \cref{prop:Aaut}.

We next show that $\alpha_{00}$ is a $\{\pm 1\}$-valued function on a set of coset representatives for $E$ in~$K_d$, and a similar argument shows that the same holds for $\alpha_{01}$ and~$\alpha_{10}$.
By definition, $A_{00}(X,Y)$ is $\{\pm 1\}$-valued on exactly one of the four values 
$\big\{X^i Y^j, X^i Y^{j+2^{d-2}}, X^{i+2^{d-2}} Y^j, X^{i+2^{d-2}} Y^{j+2^{d-2}}\big\}$
for $0 \le i < 2^{d-2}, \, 0 \le j < 2^{d-2}$.
Therefore $A_{00}(x,y^2)$ is $\{\pm 1\}$-valued on exactly one of the four values 
$\big\{x^i y^{2j}, x^i y^{2j+2^{d-1}}, x^{i+2^{d-2}} y^{2j}, x^{i+2^{d-2}} y^{2j+2^{d-1}}\big\}$
for $0 \le i < 2^{d-2}, \, 0 \le j < 2^{d-2}$,
and so $(1+x^{2^{d-2}})A_{00}(x,y^2)$ is $\{\pm 1\}$-valued on exactly one of the four values 
$\big\{x^i y^{2j}, x^i y^{2j+2^{d-1}}, x^{i+2^{d-1}} y^{2j}, x^{i+2^{d-1}} y^{2j+2^{d-1}}\big\}$
for $0 \le i < 2^{d-1}, \, 0 \le j < 2^{d-2}$. Likewise,
$y(1-x^{2^{d-2}})A_{10}(x,y^2)$ is $\{\pm 1\}$-valued on exactly one of the four values
$\big\{x^i y^{2j+1}, x^i y^{2j+2^{d-1}+1}, x^{i+2^{d-1}} y^{2j+1},$ $x^{i+2^{d-1}} y^{2j+2^{d-1}+1}\big\}$
for $0 \le i < 2^{d-1}, \, 0 \le j < 2^{d-2}$. Combining, $\alpha_{00}$ is $\{\pm 1\}$-valued on exactly one of the four values 
$\big\{x^i y^j, x^i y^{j+2^{d-1}}, x^{i+2^{d-1}} y^j, x^{i+2^{d-1}} y^{j+2^{d-1}}\big\}$
for $0 \le i < 2^{d-1}, \, 0 \le j < 2^{d-1}$.

It remains to show that in $\Z K_d$ we have
\begin{equation}
\label{eqn:alphapsi}
\alpha_{ij} \psi_{ij} \alpha_{ij}^{(-1)} = 2^{2d-2} \psi_{ij} \quad \mbox{for each $(i,j) \ne (1,1)$}.
\end{equation}
Using $x^{2^d} = 1$, for $i,k \in \{0,1\}$ we have the identity
\[
(1+x^{2^{d-1}})(1 + (-1)^i x^{2^{d-2}})(1 + (-1)^k x^{-2^{d-2}}) = 
 \begin{cases} 	2 (1+x^{2^{d-1}}) (1 + (-1)^i x^{2^{d-2}}) 	& \mbox{if $i=k$,} \\
		0 						& \mbox{if $i\ne k$,}
 \end{cases}
\]
and multiplication by $1+(-1)^j y^{2^{d-1}}$ for $j \in \{0,1\}$ then gives
\begin{equation}
\label{eqn:reln}
(1 + (-1)^i x^{2^{d-2}}) \psi_{0j} (1 + (-1)^k x^{-2^{d-2}}) = 
 \begin{cases} 	2 (1 + x^{2^{d-1}}) \chi_{ij}(x,y^2)	& \mbox{if $i=k$,} \\
		0					& \mbox{if $i\ne k$}.
 \end{cases}
\end{equation}

We can now establish \eqref{eqn:alphapsi} for $(i,j) = (0,0)$. Using \eqref{eqn:rank2recursion}, we calculate
\begin{align}
\alpha_{00} \psi_{00} \alpha_{00}^{(-1)}
 &=            \big( (1+x^{2^{d-2}})A_{00}(x,y^2) + y(1-x^{2^{d-2}})A_{10}(x,y^2) \big ) \times \psi_{00} \times \nonumber \\
 &\phantom{==} \big( (1+x^{-2^{d-2}})A_{00}(x,y^2)^{(-1)} + y^{-1}(1-x^{-2^{d-2}})A_{10}(x,y^2)^{(-1)} \big ) \nonumber \\
 &= 	       2 (1+x^{2^{d-1}}) A_{00}(x,y^2) \chi_{00}(x,y^2) A_{00}(x,y^2)^{(-1)} + \nonumber \\
 &\phantom{==} 2 (1+x^{2^{d-1}}) A_{10}(x,y^2) \chi_{10}(x,y^2) A_{10}(x,y^2)^{(-1)}, \label{eqn:indet}
\end{align}
using \eqref{eqn:reln} with $i \ne k$ to remove the terms involving $A_{00}(x,y^2) A_{10}(x,y^2)^{(-1)}$ and $A_{10}(x,y^2) A_{00}(x,y^2)^{(-1)}$, and using \eqref{eqn:reln} with $i = k$ to simplify the surviving terms. 
Take $X=x$ and $Y=y^2$ in \eqref{eqn:AuXY} to show that, in the polynomial ring $\Z[x,y]/\langle 1-x^{2^{d-1}}, 1-y^{2^d} \rangle$,
\[
A_{ij}(x,y^2) \chi_{ij}(x,y^2) A_{ij}(x,y^2)^{(-1)} = 2^{2d-4} \chi_{ij}(x,y^2) \quad \mbox{for each $(i,j) \in U_2$}.
\]
This implies that, in the polynomial ring $\Z[x,y]/\langle 1-x^{2^d}, 1-y^{2^d} \rangle$,
\begin{align*}
\lefteqn{(1+x^{2^{d-1}}) A_{ij}(x,y^2) \chi_{ij}(x,y^2) A_{ij}(x,y^2)^{(-1)}} \hspace{50mm} \\
 &= 2^{2d-4} (1+x^{2^{d-1}}) \chi_{ij}(x,y^2) \quad \mbox{for each $(i,j) \in U_2$}.
\end{align*}
Substitution in \eqref{eqn:indet} then gives
\[
\alpha_{00} \psi_{00} \alpha_{00}^{(-1)} 
 = 2^{2d-3} (1+x^{2^{d-1}}) (\chi_{00}(x,y^2) + \chi_{10}(x,y^2)) = 2^{2d-2} \psi_{00},
\]
so \eqref{eqn:alphapsi} holds for $(i,j) = (0,0)$. 

A similar derivation gives
\begin{align*}
\alpha_{01} \psi_{01} \alpha_{01}^{(-1)} 
 & = 2^{2d-3} (1+x^{2^{d-1}}) (\chi_{01}(x,y^2) + \chi_{11}(x,y^2)) = 2^{2d-2} \psi_{01}, \\
\alpha_{10} \psi_{10} \alpha_{10}^{(-1)} 
 & = 2^{2d-3} (1+y^{2^{d-1}}) (\chi_{10}(x^2,y) + \chi_{11}(x^2,y)) = 2^{2d-2} \psi_{10}, 
\end{align*}
so that \eqref{eqn:alphapsi} holds for $(i,j) = (0,1)$ and $(i,j) = (1,0)$.

Therefore the $\alpha_{ij}$ form a signature set on $K_d$ with respect to~$E$.
This shows that case $d$ is true and completes the induction.
\end{proof}

We next illustrate the recursive construction method used in the proof of \cref{thm:rank2}.

\begin{example}
\label{ex:rank2}
A trivial signature set $\{A^1_{ij}: (i,j) \in U_2\}$ on $C_2^2$ with respect to itself is given by
\[
A^1_{ij} = 1 \quad \mbox{for all $(i,j) \in U_2$}.
\]
Apply the recursion \eqref{eqn:rank2recursion} with $d=2$ to obtain the signature set $\{A^2_{ij} : (i,j) \in U_2\}$ on 
$C_4^2 = \langle x, y \rangle$ with respect to $\langle x^2, y^2 \rangle \cong C_2^2$ given by
\begin{align*}
A^2_{00} = A^2_{01} 	&= (1+x) + y(1-x) = 1+x+y-xy, \\
A^2_{10} 		&= (1+y) + x(1-y) = 1+x+y-xy, \\
A^2_{11} 		&= (1+xy) + x(1-xy) = 1+x-x^2y+xy.
\end{align*}
Apply the recursion \eqref{eqn:rank2recursion} again with $d=3$ to obtain the signature set $\{A^3_{ij} : (i,j) \in U_2\}$ on 
$C_8^2 = \langle x, y \rangle$ with respect to $\langle x^4, y^4 \rangle \cong C_2^2$ given by
\begin{align*}
A^3_{00} 
 &= (1+x^2)A^2_{00}(x,y^2) + y(1-x^2)A^2_{10}(x,y^2) \\
 &= (1+x^2)(1+x+y^2-x y^2) + y(1-x^2)(1+x+y^2- xy^2), \\[1ex]
A^3_{01} 
 &= (1+x^2)A^2_{01}(x,y^2) + y(1-x^2)A^2_{11}(x,y^2) \\
 &= (1+x^2)(1+x+y^2-x y^2) + y(1-x^2)(1+x-x^2y^2+xy^2), \\[1ex]
A^3_{10} 
 &= (1+y^2)A^2_{10}(x^2,y) + x(1-y^2)A^2_{11}(x^2,y) \\
 &= (1+y^2)(1+x^2+y-x^2y) + x(1-y^2)(1+x^2-x^4y+x^2y), \\[1ex]
A^3_{11} 
 &= (1+x^2y^2)A^2_{10}(x^2,xy) + x(1-x^2y^2)A^2_{11}(x^2,xy) \\
 &= (1+x^2y^2)(1+x^2+xy-x^3y) + x(1-x^2y^2)(1+x^2-x^5y+x^3y).
\end{align*}
\end{example}

We note that the recursion \eqref{eqn:rank2recursion} in the proof of \cref{thm:rank2} has a simpler form when expressed in terms of group ring elements $B_{ij} = A_{ij} \chi_{ij}$ and $\beta_{ij} = \alpha_{ij}\psi_{ij}$, namely
\[
\begin{array}{rl}
\beta_{00}(x,y) &= (1+x^{2^{d-1}})\big(B_{00}(x,y^2) + yB_{10}(x,y^2)\big), \\[1ex]
\beta_{01}(x,y) &= (1+x^{2^{d-1}})\big(B_{01}(x,y^2) + yB_{11}(x,y^2)\big), \\[1ex]
\beta_{10}(x,y) &= (1+y^{2^{d-1}})\big(B_{10}(x^2,y) + xB_{11}(x^2,y)\big), \\[1ex]
\beta_{11}(x,y) &= (1-y^{2^{d-1}})\big(B_{10}(x^2,xy) + xB_{11}(x^2,xy)\big).
\end{array}
\]

We now prove \cref{thm:abelian-sigset} in full generality, using the proof of \cref{thm:rank2} as a model. We abbreviate some of the proof, focussing attention on the parts for which a new argument or additional care is needed.

\begin{proof}[Proof of \cref{thm:abelian-sigset}]
The proof is by induction on $d \ge 1$.
In the case $d=1$, we have $r=2$ and ${\cal K}_{1,2} = \{C_2^2\}$. The case $d=1$ is therefore true, because there exists a trivial signature set on $C_2^2$.

Assume all cases up to $d-1 \ge 1$ are true. 
We shall write $u = (i, j, u_3, \dots, u_r) \in U_r$ as $(i, j, v)$, where $v = (u_3, \dots, u_r)$.
Let
\[
K_{d,r} = C_{2^{a_1}} \times \dots \times C_{2^{a_r}} = \langle x, y, x_3, \dots, x_r \rangle \in {\cal K}_{d,r},
\]
where $x^{2^{a_1}} = y^{2^{a_2}} = x_3^{2^{a_3}} = \dots = x_r^{2^{a_r}} = 1$ and
$d-r+2 \ge a_1 \ge a_2 \ge \dots \ge a_r \ge 1$ and $\sum_i a_i = 2d-r+2$.
The unique subgroup of $K_{d,r}$ isomorphic to $C_2^r$ is
$E_{d,r} = \langle x^{2^{a_1-1}}, y^{2^{a_2-1}}, x_3^{2^{a_3-1}}, \dots, x_r^{2^{a_r-1}} \rangle$.

If $a_r=1$, then by the inductive hypothesis there is a signature set on the group $\langle x, y, x_3, \dots, x_{r-1} \rangle\in {\cal K}_{d-1,r-1}$.
In that case we may use \cref{prop:sig-prod} to combine this with a trivial signature set on $C_2$ in order to obtain the required signature set on $K_{d,r}$ with respect to~$E_{d,r}$.

We may therefore take $d-r+2 \ge a_1 \ge a_2 \ge \dots \ge a_r \ge 2$. This implies that $r \le d$,
and if $r > 2$ then $a_3 \le d-r+1$ (otherwise 
$2d-r+2 = \sum_i a_i \ge 3(d-r+2) + (r-3)2 = 3d-r$, giving the contradiction $r \le d \le 2$).
By the inductive hypothesis, the group 
\[
C_{2^{a_1-1}} \times C_{2^{a_2-1}} \times C_{2^{a_3}} \times \dots \times C_{2^{a_r}}  = \langle X, Y, x_3, \dots, x_r \rangle \in {\cal K}_{d-1,r},
\]
where $X^{2^{a_1-1}} = Y^{2^{a_2-1}} = x_3^{2^{a_3}} = \dots = x_r^{2^{a_r}} = 1$,
therefore contains a signature set $\{A_{ijv} : (i,j,v) \in U_r\}$ with respect to 
$E_{d-1,r} = \langle X^{2^{a_1-2}}, Y^{2^{a_2-2}}, x_3^{2^{a_3-1}}, \dots, x_r^{2^{a_r-1}} \rangle$.

Regard each group ring element $A_{ijv}$ as a polynomial in $X,Y,x_3,\dots,x_r$, but abbreviate this as $A_{ijv}(X,Y)$ because we will make variable substitutions only for $X$ and~$Y$. Similarly, regard each character of $E_{d-1,r}$ as a polynomial
\[
\chi_{ijv}(X,Y) = \big(1+(-1)^i X^{2^{a_1-2}}\big) \big(1+(-1)^j Y^{2^{a_2-2}})\, \tau_v 
\]
where
\[
\tau_v = (1+(-1)^{u_3}x_3^{2^{a_3-1}}) \dots (1+(-1)^{u_r}x_r^{2^{a_r-1}}).
\]
By assumption, in the polynomial ring $\Z[X,Y,x_3,\dots,x_r]/\langle 1-X^{2^{a_1-1}}, 1-Y^{2^{a_2-1}}, 1-x_3^{2^{a_3}}, \dots, 1-x_r^{2^{a_r}} \rangle$ we have
\begin{equation}
\label{eqn:AuXY-gen}
A_{ijv}(X,Y) \chi_{ijv}(X,Y) A_{ijv}(X,Y)^{(-1)} = 2^{2d-2r} \chi_{ijv}(X,Y) \quad \mbox{for each $(i,j,v) \in U_r$}.
\end{equation}

We wish to construct a signature set $\{\alpha_{ijv}: (i,j,v) \in U_r\}$ on $K_{d,r}$ with respect to~$E_{d,r}$.
Define the $\alpha_{ijv}$ in $\Z K_{d,r}$ in terms of the polynomials~$A_{ijv}$ via
\begin{equation}
\label{eqn:genrecursion}
\left .
\begin{array}{rl}
\alpha_{00v} &= (1+x^{2^{a_1-2}})A_{00v}(x,y^2) + y(1-x^{2^{a_1-2}})A_{10v}(x,y^2), \\[1ex]
\alpha_{01v} &= (1+x^{2^{a_1-2}})A_{01v}(x,y^2) + y(1-x^{2^{a_1-2}})A_{11v}(x,y^2), \\[1ex]
\alpha_{10v} &= (1+y^{2^{a_2-2}})A_{10v}(x^2,y) + x(1-y^{2^{a_2-2}})A_{11v}(x^2,y), \\[1ex]
\alpha_{11v} &= (1+x^{2^{a_1-2}}y^{2^{a_2-2}})A_{10v}(x^2,x^{2^{a_1-a_2}}y) + x(1-x^{2^{a_1-2}}y^{2^{a_2-2}})A_{11v}(x^2,x^{2^{a_1-a_2}}y),
\end{array}
\right \} 
\end{equation}
and let the characters of $E_{d,r}$ be
\[
\psi_{ijv} = \big (1+(-1)^i x^{2^{a_1-1}} \big) \big (1+(-1)^j y^{2^{a_2-1}} \big) \,\tau_v \quad \mbox{for each $(i,j,v) \in U_r$}.
\]

We firstly use \cref{prop:Aaut} to show it is sufficient to prove for each $(i,j,v) \ne (1,1,v)$ that $\alpha_{ijv}$ is a signature block with respect to~$\psi_{ijv}$.
Let $\sigma$ be the group automorphism of $K_{d,r}$ that maps $x$ to itself and maps $y$ to $x^{2^{a_1-a_2}}y$ (which has order~$2^{a_2}$).
Then $\sigma(\alpha_{10v}) = \alpha_{11v}$ by definition, and $\sigma$ fixes $E_{d,r}$, and
$\sigma(\psi_{10v}) = \psi_{11v}$.
Therefore if $\alpha_{10v}$ is a signature block on $K_{d,r}$ with respect to $\psi_{10v}$, then $\alpha_{11v}$ is a signature block on $K_{d,r}$ with respect to~$\psi_{11v}$ by \cref{prop:Aaut}.

We next show that each $\alpha_{00v}$ is a $\{\pm 1\}$-valued function on a set of coset representatives for $E_{d,r}$ in~$K_{d,r}$, and a similar argument shows that the same holds for each $\alpha_{01v}$ and~$\alpha_{10v}$. Fix $z = x_3^{i_3} \dots x_r^{i_r}$.
By definition, $A_{00v}(X,Y)$ is $\{\pm 1\}$-valued on 
exactly one of the four values 
$\big\{X^i Y^j z, X^i Y^{j+2^{a_2-2}} z, X^{i+2^{a_1-2}} Y^j z,$ $ X^{i+2^{a_1-2}} Y^{j+2^{a_2-2}} z\big\}$
for $0 \le i < 2^{a_1-2}, \, 0 \le j < 2^{a_2-2}$.
It follows that $\alpha_{00v}$ is $\{\pm 1\}$-valued on exactly one of the four values 
$\big\{x^i y^j z, x^i y^{j+2^{a_2-1}} z, x^{i+2^{a_1-1}} y^j z, x^{i+2^{a_1-1}} y^{j+2^{a_2-1}} z\big\}$
for $0 \le i < 2^{a_1-1}, \, 0 \le j < 2^{a_2-1}$.

It remains to show that in $\Z K_{d,r}$ we have
\begin{equation}
\label{eqn:alphapsi-gen}
\alpha_{ijv} \psi_{ijv} \alpha_{ijv}^{(-1)} = 2^{2d-2r+2} \psi_{ijv} \quad \mbox{for each $(i,j,v) \ne (1,1,v)$}.
\end{equation}
For $i,j,k \in \{0,1\}$, we have the identity
\begin{equation}
\label{eqn:reln-gen}
(1 + (-1)^i x^{2^{a_1-2}}) \psi_{0jv} (1 + (-1)^k x^{-2^{a_1-2}}) = 
 \begin{cases} 	2 (1 + x^{2^{a_1-1}}) \chi_{ijv}(x,y^2)	& \mbox{if $i=k$,} \\
		0					& \mbox{if $i\ne k$},
 \end{cases}
\end{equation}
from which we now establish \eqref{eqn:alphapsi-gen} for $(i,j,v) = (0,0,v)$. We calculate
\begin{align}
\alpha_{00v} \psi_{00v} \alpha_{00v}^{(-1)}
 &=            \big( (1+x^{2^{a_1-2}})A_{00v}(x,y^2) + y(1-x^{2^{a_1-2}})A_{10v}(x,y^2) \big ) \times \psi_{00v} \times \nonumber \\
 &\phantom{==} \big( (1+x^{-2^{a_1-2}})A_{00v}(x,y^2)^{(-1)} + y^{-1}(1-x^{-2^{a_1-2}})A_{10v}(x,y^2)^{(-1)} \big ) \nonumber \\
 &= 	       2 (1+x^{2^{a_1-1}}) A_{00v}(x,y^2) \chi_{00v}(x,y^2) A_{00v}(x,y^2)^{(-1)} + \nonumber \\
 &\phantom{==} 2 (1+x^{2^{a_1-1}}) A_{10v}(x,y^2) \chi_{10v}(x,y^2) A_{10v}(x,y^2)^{(-1)}, \label{eqn:indet-gen}
\end{align}
using \eqref{eqn:reln-gen}.
Take $X=x$ and $Y=y^2$ in \eqref{eqn:AuXY-gen} to show that, in the polynomial ring $\Z[x,y,x_3,\dots,x_r]/\langle 1-x^{2^{a_1}}, 1-y^{2^{a_2}}, 1-x_3^{2^{a_3}}, \dots, 1-x_r^{2^{a_r}} \rangle$,
\begin{align*}
\lefteqn{(1+x^{2^{a_1-1}}) A_{ijv}(x,y^2) \chi_{ijv}(x,y^2) A_{ijv}(x,y^2)^{(-1)}} \hspace{45mm} \\
 & = 2^{2d-2r} (1+x^{2^{a_1-1}}) \chi_{ijv}(x,y^2) \quad \mbox{for each $(i,j,v) \in U_r$}.
\end{align*}
Substitution in \eqref{eqn:indet-gen} then gives
\[
\alpha_{00v} \psi_{00v} \alpha_{00v}^{(-1)} 
 = 2^{2d-2r+1} (1+x^{2^{a_1-1}}) (\chi_{00v}(x,y^2) + \chi_{10v}(x,y^2)) = 2^{2d-2r+2} \psi_{00v},
\]
so \eqref{eqn:alphapsi-gen} holds for $(i,j,v) = (0,0,v)$. 

A similar derivation gives
\begin{align*}
\alpha_{01v} \psi_{01v} \alpha_{01v}^{(-1)} 
 & = 2^{2d-2r+1} (1+x^{2^{a_1-1}}) (\chi_{01v}(x,y^2) + \chi_{11v}(x,y^2)) = 2^{2d-2r+2} \psi_{01v}, \\
\alpha_{10v} \psi_{10v} \alpha_{10v}^{(-1)} 
 & = 2^{2d-2r+1} (1+y^{2^{a_2-1}}) (\chi_{10v}(x^2,y) + \chi_{11v}(x^2,y)) = 2^{2d-2r+2} \psi_{10v}, 
\end{align*}
so that \eqref{eqn:alphapsi-gen} holds for $(i,j,v) = (0,1,v)$ and $(i,j,v) = (1,0,v)$.

Therefore the $\alpha_{ijv}$ form a signature set on $K_{d,r}$ with respect to~$E_{d,r}$.
This shows that case $d$ is true and completes the induction.
\end{proof}

We now illustrate the recursive construction method used in the proof of \cref{thm:abelian-sigset}. 
\begin{example}
\label{ex:sigset844}
We shall construct a signature set on $C_8 \times C_4^2$.
By \cref{ex:rank2}, there is a signature set $\{A'_{ik} : (i,k) \in U_2\}$ on $C_4^2 = \langle x, z \rangle$ with respect to $\langle x^2, z^2 \rangle$ given by 
\begin{align*}
A'_{00} = A'_{01} = A'_{10} &= 1+x+z-xz, \\
A'_{11} &= 1+x-x^2z+xz.
\end{align*}
Use the product construction of \cref{prop:sig-prod} to combine this with a trivial signature set on $C_2$, producing a signature set $\{A_{ijk}: (i,j,k) \in U_3\}$ on $C_4 \times C_2 \times C_4 = \langle x,y,z \rangle$ with respect to $\langle x^2,y,z^2 \rangle \cong C_2^3$ given by
\begin{align*}
A_{000} = A_{010} = A_{001} = A_{011} = A_{100} = A_{110} &= 1+x+z-x z, \\
A_{101} = A_{111} &= 1+x-x^2z+x z.
\end{align*}
Now apply the recursion \eqref{eqn:genrecursion} to produce a signature set 
$\{\alpha_{ijk}: (i,j,k) \in U_3\}$ on 
$C_8 \times C_4^2 = \langle x,y,z \rangle$ with respect to $\langle x^4,y^2,z^2 \rangle \cong C_2^3$ given by
\begin{align*}
\alpha_{000} 
 &= (1+x^2)(1+x+z-x z) + y(1-x^2)(1+x+z-x z), \\
\alpha_{001} 
 &= (1+x^2)(1+x+z-x z) + y(1-x^2)(1+x-x^2z+x z), \\
\alpha_{010} 
 &= (1+x^2)(1+x+z-x z) + y(1-x^2)(1+x+z-x z), \\
\alpha_{011} 
 &= (1+x^2)(1+x+z-x z) + y(1-x^2)(1+x-x^2z+x z), \\
\alpha_{100} 
 &= (1+y)(1+x^2+z-x^2 z)+ x(1-y)(1+x^2+z-x^2 z), \\
\alpha_{101} 
 &= (1+y)(1+x^2-x^4z+x^2 z)+ x(1-y)(1+x^2-x^4z+x^2 z), \\
\alpha_{110} 
 &= (1+x^2y)(1+x^2+z-x^2 z) + x(1-x^2y)(1+x^2+z-x^2 z), \\
\alpha_{111} 
 &= (1+x^2y)(1+x^2-x^4z+x^2 z) + x(1-x^2y)(1+x^2-x^4z+x^2 z) .
\end{align*}
\end{example}

\section{Further Construction Methods}
\label{sec:further}

As shown in \cref{tab:2-groups}, our main result (\cref{thm:main}) uses signature sets on abelian groups to provide constructions for difference sets in the great majority of the groups of order 64 and 256 that are not excluded by Theorems~\ref{thm:turyn} and~\ref{thm:dillon}.
In this section, we describe the methods that were used to show that the 22 remaining groups of order 64, and the 1,416 remaining groups of order 256, all belong to~${\cal H}$. 

In \cref{subsec:nonabelian-sigset}, we present a construction method arising under \cref{thm:prehadamard} from a signature set on a nonabelian group; recall that \cref{defn:sig} for a signature set does not require the group $K$ to be abelian. 
In \cref{subsec:productPTA}, we present a product construction using perfect ternary arrays, without constraining these arrays in relation to a subgroup.
In \cref{subsec:transfer}, we describe three non-systematic methods of transferring a difference set in one group to another.
We used the methods of Sections~\ref{subsec:nonabelian-sigset}--\ref{subsec:transfer} to establish that all but one of the 22 remaining non-excluded groups of order 64, and all but one of the 1,416 remaining non-excluded groups of order 256, belong to~${\cal H}$.
In \cref{subsec:final}, we describe the construction of a Hadamard difference set in both of these final groups using group representations.
In \cref{subsec:combination}, we show that the signature set construction of \cref{subsec:nonabelian-sigset} and the perfect ternary array product construction of~\cref{subsec:productPTA} are closely related and can sometimes be combined, which could in future assist in determining which $2$-groups of order larger than~256 belong to~${\cal H}$.

\subsection{Signature Set on Nonabelian Group}
\label{subsec:nonabelian-sigset}

Our first construction method applies \cref{thm:prehadamard} to a signature set on a nonabelian group to produce Hadamard difference sets in a variety of larger groups. We illustrate this method by exhibiting a signature set on the quaternion group of order~$8$.

\begin{prop}
\label{prop:sigsetQ}
Let $Q = \langle x,y : x^4=y^4=1, \, yxy^{-1} = x^{-1}, \, x^2=y^2 \rangle$ be the quaternion group of order~$8$, and let $G$ be a group of order $16$ containing a subgroup isomorphic to~$Q$. Then $G \in {\cal H}$.
\end{prop}
\begin{proof}
Let $E_1 = \langle x^2 \rangle \cong C_2$, and let 
\[
\chi_0 = 1+x^2, \quad \chi_1 = 1-x^2 
\]
be the characters of~$E_1$.
Since $E_1$ is the unique subgroup of $Q$ isomorphic to $C_2$, and $Q$ has index $2$ and so is normal in $G$, we have that $E_1$ is normal in~$G$.
Therefore by \cref{thm:prehadamard} with $r=1$, it is sufficient to exhibit a signature set $\{A_0, A_1\}$ on~$Q$ with respect to~$E_1$ (and then according to \eqref{eqn:Dast2} there is a difference set in $G$ of the form $g_0A_0\chi_0 + g_1A_1\chi_1$).

Let $A = 1-x-y-xy$, and let $\{A_0, A_1\} = \{A, A\}$. 
Then $A$ is a $\{\pm1\}$-valued function on a set of coset representatives for $E_1$ in $Q$, and direct calculation shows that $A A^{(-1)} = 4$ in $\Z Q$. Since $E_1$ is a central subgroup of $Q$, we therefore have in $\Z Q$ that
\[
A_u \chi_u A_u^{(-1)} = A_u A_u^{(-1)} \chi_u = 4 \chi_u = \tfrac{|Q|}{2}\chi_u \quad \mbox{for $u \in \{0,1\}$},
\]
as required.
\end{proof}
As noted prior to \cref{tab:2-groups}, we can use \cref{thm:main} and \cref{prop:sigsetQ} to recover the classification of Hadamard groups of order~16 carried out in the 1970s: \cref{thm:main} accounts for the 10 groups 
containing a normal subgroup isomorphic to $C_2^2$, and \cref{prop:sigsetQ} accounts for 2 further groups (the generalized quaternion group and the semidihedral group) containing a subgroup isomorphic to~$Q$.

Furthermore, using \cref{prop:sig-prod} we may now take the product of a signature set on $Q$ with respect to $E_1$ given in the proof of \cref{prop:sigsetQ}, and a trivial signature set on $C_2$, to give a signature set on $Q \times C_2$ with respect to $E_1 \times C_2 \cong C_2^2$. 
Then from \cref{thm:prehadamard}, every group of order $64$ containing a normal subgroup isomorphic to $Q \times C_2$ belongs to~${\cal H}$. 

We now use a Hadamard difference set to construct a signature set on certain 
groups of order $2^{2d+1}$.

\begin{prop}
\label{prop:HDSbyC2}
Suppose $D$ is a Hadamard difference set in a group~$H$, and let $E_1 \cong C_2$.
Then $\{D,D\}$ is a signature set on $H \times E_1$ with respect to~$E_1$. 
\end{prop}
\begin{proof}
We are given that $D$ is a $\{\pm 1\}$-valued function on the set $H$ of coset representatives for $E_1$ in $H \times E_1$. 
Let $\{A_0, A_1\} = \{D, D\}$, and write $E_1 = \langle x \rangle$ so that the characters of $E_1$ are $\chi_0 = 1+x$ and $\chi_1 = 1-x$. 
Since $x$ commutes with $D$, we have in $\Z (H \times E_1)$ that
\[
A_u \chi_u A_u^{(-1)} 
= D D^{(-1)} \chi_u = |H| \chi_u = \tfrac{|H \times E_1|}{2}\chi_u \quad \mbox{for $u \in \{0,1\}$},
\]
as required.
\end{proof}

\begin{cor}
\label{cor:HDSbyC2}
Suppose $H \in {\cal H}$. Let $G$ be a group containing a normal subgroup $E_1 \cong C_2$, and containing $H \times E_1$ as a subgroup of index~$2$. Then $G \in {\cal H}$.
\end{cor}
\begin{proof}
By \cref{prop:HDSbyC2}, there exists a signature set on $H \times E_1$ with respect to~$E_1$. Since $E_1$ and $H \times E_1$ are both normal in $G$, we have $G \in {\cal H}$ by \cref{thm:prehadamard}.
\end{proof}

The technique of constructing Hadamard difference sets from signature sets on nonabelian groups appears to have significant potential, but we do not currently have a method of producing such signature sets that is as powerful as the recursive construction used to prove \cref{thm:abelian-sigset} for abelian groups.

\subsection{Product of Perfect Ternary Arrays}
\label{subsec:productPTA}

Our second construction method relies on a key feature of the proof of \cref{prop:sigsetQ}, namely the existence of a $\{+1, 0, -1\}$-valued function~$A$ on the group $Q$ satisfying $AA^{(-1)} = 4$ in $\Z Q$. This function $A$ is also $\{\pm 1\}$-valued on a set of coset representatives for a subgroup of~$Q$, but we do not require this additional structure in the following definition.

\begin{defn}
\label{defn:pta}
Let $G$ be a group.
A \emph{perfect ternary array} in $G$ is a $\{+1, 0, -1\}$-valued function $T$ on $G$ satisfying $TT^{(-1)} = c$ in $\Z G$ for some integer~$c$. 
\end{defn}

The set of elements of a group $G$ on which a group ring element $A \in \Z G$ is nonzero is the \emph{support} of~$A$; the size of this set is the \emph{weight} of~$A$, written $\wt(A)$.
We firstly show that the integer $c$ in \cref{defn:pta} is equal to the weight of the perfect ternary array, and that it is a square.

\begin{lemma}
\label{lem:PTAwt}
Let $G$ be a group, and suppose $T = \displaystyle{\sum_{g \in G} t_g g}$ is a perfect ternary array where each $t_g \in \{+1, 0, -1\}$.
Then $TT^{(-1)} = \wt(T) = \displaystyle{\Big(\sum_{g \in G} t_g\Big)^2}$.
\end{lemma}

\begin{proof}
For some integer $c$, we have
\[
c = TT^{(-1)} 
  = \Big(\sum_{h \in G} t_h h\Big) \Big(\sum_{g \in G} t_g g^{-1} \Big) 
  = \sum_{k \in G} \Big( \sum_{g \in G} t_{kg} t_g \Big) k
\]
by writing $k = hg^{-1}$. Comparison of the coefficients of $1_G$ and $k \ne 1_G$
gives
\begin{align}
c &= \sum_{g \in G} t_g^2, \label{eqn:ctg2} \\
0 &= \sum_{g \in G} t_{kg} t_g \quad \mbox{for $k \ne 1_G$}. \nonumber
\end{align}
These relations together give
\[
c = \sum_{k \in G} \sum_{g \in G} t_{kg} t_g 
  = \sum_{g \in G} \Big(\sum_{h \in G} t_h\Big) t_g 
  = \Big(\sum_{g \in G} t_g\Big)^2.
\]
The result follows by combining with \eqref{eqn:ctg2}, noting that
$\sum_{g\in G}t_g^2 = \wt(T)$ because $T$ is $\{+1,0,-1\}$-valued.
\end{proof}
\noindent
By \cref{lem:PTAwt} and \eqref{eqn:DDstar}, we may regard a Hadamard difference set in a group $G$ as a perfect ternary array $T$ in $G$ for which $TT^{(-1)} = |G|$.
A survey of results on the matrix representation of a perfect ternary array in an abelian group is given in~\cite{PTAs1999}. 
We next give two examples of perfect ternary arrays of weight~$4$, whose properties can be verified by direct calculation. The second example appears in the proof of \cref{prop:sigsetQ}.

\begin{example}[Dillon 1990 (unpublished)]
\label{ex:pta-mod2}
\begin{enumerate}[$(i)$]
\item
Suppose $G$ is a group containing a nonidentity element~$x$ and an involution (element of order~$2$) $y$ that commutes with~$x$.
Then $T = 1-x-y-xy$ is a perfect ternary array of weight~$4$ in~$G$.

\item
Let $Q = \langle x,y : x^4=y^4=1, \, yxy^{-1} = x^{-1}, \, x^2=y^2 \rangle$ be the quaternion group of order~$8$. Then $T= 1-x-y-xy$ is a perfect ternary array of weight~$4$ in~$Q$.
\end{enumerate}
\end{example}
\noindent
Every perfect ternary array of weight~$4$ in a group of even order is equivalent to \cref{ex:pta-mod2}~$(i)$ or~$(ii)$~\cite[Lemma~2]{bhattsmith}. 

We now construct a Hadamard difference set as a product of perfect ternary arrays.

\begin{prop}[Dillon 1990 (unpublished), Bhattacharya and Smith \cite{bhattsmith}]
\label{prop:productPTAgivesDS}
Let $T_1, T_2, \dots, T_s$ be subsets of a group $G$, and let $D = \prod_{i=1}^s T_i$. Suppose that
\begin{enumerate}[$(i)$]
\item
each $T_i$ is a perfect ternary array in $G$,
\item
$\wt(D) = \prod_{i=1}^s \wt(T_i)$,
\item
$\wt(D) = |G|$.
\end{enumerate}
Then $D$ corresponds to a Hadamard difference set in~$G$. 
\end{prop}

\begin{proof}
By condition $(ii)$, $D$ is a $\{+1, 0, -1\}$-valued function on~$G$.
Now $DD^{(-1)} = \prod_{i=1}^s \wt(T_i)$ by \cref{lem:PTAwt}, 
and then by conditions $(ii)$ and $(iii)$ we have $DD^{(-1)} = |G|$.
\end{proof}
\noindent
Since a Hadamard difference set is a special case of a perfect ternary array, we may regard \cref{thm:productconstruction} as constructing a Hadamard difference set in $G$ as the product $D_1 D_2$ of two perfect ternary arrays $D_1$ and $D_2$ contained in subgroups $H_1$ and $H_2$ of~$G$. 
In contrast, \cref{prop:productPTAgivesDS} constructs Hadamard difference sets as the product of $s$ perfect ternary arrays $T_i$, with the important relaxation that each $T_i$ need not be structurally constrained in relation to a subgroup of~$G$.

This generality gives \cref{prop:productPTAgivesDS} considerable power.
We take each $T_i$ to be either a perfect ternary array of weight~$4$ (having one of the two forms of \cref{ex:pta-mod2}), or else a Hadamard difference set in a subgroup of~$G$.
This allows us to construct all 27 inequivalent difference sets in the 12 groups of order 16 contained in ${\cal H}$ \cite{bhattsmith}; a difference set in 17 of the 22 remaining non-excluded groups of order~64; 
and a difference set in 1,358 of the 1,416 remaining non-excluded groups of order~256 (see \cref{tab:2-groups}). 
However, the same generality means that testing whether a group $G$ lies in ${\cal H}$ because of \cref{prop:productPTAgivesDS} (involving a computer search over all suitable perfect ternary arrays) is significantly slower than testing whether $G$ lies in~${\cal H}$ because of \cref{thm:main} (involving simply testing whether $G$ contains a suitable normal abelian subgroup); see \cref{sec:verification} for further details.

\subsection{Transfer Methods}
\label{subsec:transfer}

The construction methods of previous sections are collectively sufficient to demonstrate that the great majority of the groups of order 64 and 256 that are not excluded by Theorems~\ref{thm:turyn} and~\ref{thm:dillon} belong to~${\cal H}$. The key in almost all of these demonstrations is the existence of a signature set on a normal subgroup, from which a difference set arises using \cref{thm:prehadamard}.
Nonetheless, while the signature set concept is very powerful, it does not appear to be sufficient to determine ${\cal H}$ completely. The reason is that some groups (2 of order 64, and 10 of order 256) have the property that each of their normal subgroups also occurs as a normal subgroup of a group that is not in~${\cal H}$. 
We therefore require construction methods that do not rely on a signature set. We now describe three such methods, each of which uses a difference set in one group to discover a difference set in another (and so ``transfers'' a difference set between the two groups).

The first transfer method makes use of the equivalence between
a difference set in a group $G$ and a symmetric design on whose points $G$ acts as a regular (sharply transitive) automorphism group.
If the full automorphism group of the design is sufficiently large, it may well contain other subgroups which also act regularly on the points of the design; in this case, each of these subgroups also contains a difference set.
For example, the group $C_2^4$ contains a difference set giving a $(16,6,2)$ symmetric design whose 2-rank is~$6$, and the automorphism group of this design contains 12 nonisomorphic subgroups of order 16 acting regularly on the points of the design. We thereby transfer a single difference set in $C_2^4$ to a difference set in all 11 of the other Hadamard groups of order~16.
Similarly, the group $C_2^6$ contains a difference set giving a $(64,28,12)$ symmetric design whose 2-rank is $8$, and the automorphism group of this design contains 171 nonisomorphic subgroups of order 64 acting regularly on the points of the design. We thereby transfer a single difference set in $C_2^6$ to 170 of the other 258 Hadamard groups of order~64.

The second transfer method applies when a difference set gives an algebraic structure in the group ring that also exists in other group rings.
An example is Dillon's proof \cite{dillon} of \cref{thm:dillon}, which transfers a putative difference set in a group with a large dihedral quotient to a difference set in a group with a large cyclic quotient in order to apply the nonexistence result of \cref{thm:turyn}.
A second example is \cref{thm:prehadamard}, which can be viewed as using \cref{thm:drisko} to transfer a difference set in an abelian group that contains $K$ to a difference set in a variety of nonabelian groups containing $K$.
A third example is \cite[Thm.~2]{dillon2groups}, which transfers a difference set among groups sharing a subgroup $H$ of index~$2$ and a central element $g$ not in~$H$.
In general, suppose that a group $G$ is known to contain a difference set $D$, and that $G$ contains a large normal subgroup~$K$. Let $\{g_u\}$ be a set of coset representatives for $K$ in $G$, and partition the elements of $D$ according to their membership of the cosets of $K$ to write $D = \sum_u g_u D_u$, where each $D_u \in \Z K$.
Now let $G'$ be a group having the same order as $G$ and containing a normal subgroup $K'$ isomorphic to $K$. Let $\phi$ be an isomorphism from $K$ to $K'$. 
To transfer the difference set $D$ from $G$ to $G'$ we seek, by hand or by computer search, a set of coset representatives $\{g'_u\}$ for $K'$ in $G'$ for which $\sum_u g'_u \phi(D_u)$ is a difference set in~$G'$. 

The third transfer method takes advantage of the structure created by the GAP method for labelling group elements.
A difference set $D$ with parameters $(v,k,\lambda)$ in a group $G$ of order $v$ can be represented in GAP as a $k$-subset $S$ of the element labels $\{1,2,\dots,v\}$.
Given such a subset $S$ representing a difference set in $G$, we test in GAP whether the same subset $S$ also represents a difference set in another group $G'$ of order~$v$. This method appears to have a greater chance of success when the GAP numbering for $G'$ is close to that of $G$, which often occurs when $G'$ has similar structure to~$G$.

None of these three transfer methods is systematic, and it is not yet clear when they can be expected to succeed. Nonetheless, we were able to apply them to show that all but one of the remaining 5 non-excluded groups of order 64, and all but one of the remaining 58 non-excluded groups of order 256, belong to ${\cal H}$ (see \cref{tab:2-groups}). 
We construct a difference set in the final group of order 64 and of order 256 in \cref{subsec:final}.

\subsection{The Final Group of Order 64 and of Order 256}
\label{subsec:final}

The final two groups whose membership in ${\cal H}$ we wish to demonstrate are the order 64 modular group 
\[
M_{64} = C_{32} \rtimes_{17} C_2 = \langle x,y : x^{32} = y^2 = 1, \, yxy^{-1} = x^{17} \rangle,
\]
and the order 256 group 
\[
C_{64} \rtimes_{47} C_4 = \langle x, y : x^{64}=y^4=1, \, yxy^{-1}=x^{47}\rangle
\]
that is referenced in \cite{gap} as SmallGroup$(256,536)$.
These nonabelian groups are each a cyclic extension of a cyclic group, and have small center and high exponent.
Historically, they were the last groups of their order whose membership in ${\cal H}$ was determined: $M_{64}$ in 1991~\cite{smithliebler}, and SmallGroup$(256,536)$ in 2016~\cite{yolland}. 

We firstly describe the original construction method used in \cite{smithliebler} and~\cite{yolland}, which strengthens the representation theory approach used in \cite[Sect.~5]{davis-smith94} to construct a difference set in the order 256 group $C_{64} \rtimes_{33} C_4 = \langle x,y : x^{64}=y^4=1, \, yxy^{-1}=x^{33} \rangle$.
We shall then reinterpret these constructions as arising from a modification of a signature set. 

\begin{prop}
\label{prop:original-final}
Let $G$ be a $2$-group, let $g$ be a central involution in $G$, and
let $\natural$ be the natural map from $G$ onto~$G/\langle g \rangle$.
Suppose there are $\{+1,0,-1\}$-valued functions $D_0, D_1$ on $G$ for which $D_0(1+g)$ and $D_1(1-g)$ have disjoint supports whose union is~$G$, and for which
\begin{align}
\natural(D_0) \natural(D_0)^{(-1)} 
 &= \tfrac{|G|}{4} \quad \mbox{in $\Z(G/\langle g \rangle)$}, \label{eqn:D0D0} \\
D_1(1-g)D_1^{(-1)} 
 &= \tfrac{|G|}{4}(1-g) \quad \mbox{in $\Z G$}. \label{eqn:D1D1} 
\end{align}
Then $G \in {\cal H}$.
\end{prop}
\begin{proof}
We note that the existence of a central involution $g$ in the $2$-group $G$ follows from the class equation for finite groups. Let
\begin{equation}
\label{eqn:D-defn}
D = D_0(1+g)+D_1(1-g) \quad \mbox{in $\Z G$},
\end{equation}
which is a $\{\pm 1\}$-valued function on~$G$ by the assumption on the supports of $D_0(1+g)$ and~$D_1(1-g)$.

We now calculate
\begin{equation}
\label{eqn:original}
DD^{(-1)} = 2D_0(1+g)D_0^{(-1)} +2D_1(1-g)D_1^{(-1)} \quad \mbox{in $\Z G$}.
\end{equation}
By \eqref{eqn:D0D0}, in $\Z(G/\langle g \rangle)$ we have
\[
\tfrac{|G|}{4}1_{G/\langle g \rangle} = \natural(D_0) \natural(D_0)^{(-1)} 
 = (D_0 \langle g \rangle)(D_0^{(-1)} \langle g \rangle) = D_0 D_0^{(-1)} \langle g \rangle,
\]
so that in $\Z G$ we have
\[
\tfrac{|G|}{4}(1+g) = D_0 D_0^{(-1)} (1+g) = D_0 (1+g) D_0^{(-1)}
\]
because $g$ is central in~$G$.
Substitute this and \eqref{eqn:D1D1} into \eqref{eqn:original} to obtain
\[
DD^{(-1)} = \tfrac{|G|}{2}(1+g) + \tfrac{|G|}{2}(1-g) = |G|.
\]
Therefore $D$ corresponds to a Hadamard difference set in~$G$.
\end{proof}
\noindent
When applying \cref{prop:original-final}, we firstly seek a $\{+1,0,-1\}$-valued group ring element $D_0$ satisfying condition~\eqref{eqn:D0D0}, namely that $\natural(D_0)$ is a perfect ternary array of weight $\tfrac{|G|}{4}$ in the factor group~$G/\langle g\rangle$. We then seek a $\{+1, 0, -1\}$-valued group ring element~$D_1$ satisfying~\eqref{eqn:D1D1} for which $D_0(1+g)$ and $D_1(1-g)$ have disjoint supports whose union is~$G$. 
It turns out that finding $D_0$ is relatively easy, whereas finding $D_1$ is much more difficult.

\begin{example}[Liebler and Smith construction for $M_{64}$ \cite{smithliebler}] 
\label{ex:modular} 
We apply \cref{prop:original-final} to construct a Hadamard difference set in 
$M_{64} = C_{32} \rtimes_{17} C_2 = \langle x,y : x^{32} = y^2 = 1,\ yxy^{-1} = x^{17} \rangle$. 
The center of $M_{64}$ is $\langle x^2\rangle$, so $x^{16}$ is a central involution. 

A $\{+1,0,-1\}$-valued group ring element $D_0$ satisfying
\[
\natural(D_0) \natural(D_0)^{(-1)} = 16 \quad \mbox{in $\Z(M_{64}/\langle x^{16} \rangle)$}
\]
is given by 
\[ 
D_0 = A_{00}(1+y) + A_{01}(1-y),
\]
where
\begin{align*}
A_{00} &= -x^7 (1+x^8)+(1-x^8), \\
A_{01} &= x(1+x^8)+x^4(1-x^8). 
\end{align*}
This was easily found by hand, because the factor group $M_{64}/\langle x^{16} \rangle$ is isomorphic to the abelian group $C_{16} \times C_2$.

A $\{+1,0,-1\}$-valued group ring element $D_1$ satisfying
\[
D_1(1-x^{16})D_1^{(-1)} = 16(1-x^{16}) \quad \mbox{in $\Z M_{64}$}
\]
is given by
\[ 
D_1 = A_{10}(1+y) +A_{11}(1-y),
\]
where
\begin{align*}
A_{10} &= (x^6-x^5)(1-x^8),       \\
A_{11} &= (x^2+x^3)(1+x^8).
\end{align*}
This was found by hand using the irreducible representations induced by the character (homomorphism) that maps $x^{16}$ to~$-1$.

Now $D_0(1+x^{16})$ has support $(1+x+x^4+x^7)\langle x^8, y \rangle$,
and $D_1(1-x^{16})$ has support $(x^2+x^3+x^5+x^6) \langle x^8, y \rangle$.
These supports are disjoint and their union is~$M_{64}$.
We conclude from the construction of \cref{prop:original-final} that 
$D= D_0(1+x^{16})+D_1(1-x^{16})$ corresponds to a difference set in~$M_{64}$.
\end{example}
 
\begin{example}[Yolland construction for SmallGroup$(256,536)$ \cite{yolland}]
\label{ex:yolland}
We apply \cref{prop:original-final} to construct a Hadamard difference set in 
$G = C_{64} \rtimes_{47} C_4 = \langle x, y : x^{64}=y^4=1, \,  yxy^{-1}=x^{47} \rangle$. 
The center of $G$ is $\langle x^{32} \rangle$, so $x^{32}$ is a central involution. 

A $\{+1,0,-1\}$-valued group ring element $D_0$ satisfying
\[
\natural(D_0) \natural(D_0)^{(-1)} = 64 \quad \mbox{in $\Z(G/\langle x^{32} \rangle)$}
\]
is given by 
\[ 
D_0 = A_{00}(1+y^2)+A_{01}(1-y^2), 
\]
where
\begin{align*}
A_{00} &= \big( (1-x^8)-x^2(1+x^8) \big) (1+x^{16})+(x^5+x^6y)(1+x^8)(1-x^{16}), \\
A_{01} &= \big( (1-x^8)-x^5(1+x^8) \big) y(1+x^{16})+\big( -x^3(1-x^8)y +x^3(1+x^8) \big)(1-x^{16}).
\end{align*}
This was found by hand by seeking a perfect ternary array of weight~$64$ in the nonabelian factor group 
$G/\langle x^{32} \rangle \cong C_{32} \rtimes_{15} C_4$.

A $\{+1,0,-1\}$-valued group ring element $D_1$ satisfying
\[
D_1(1-x^{32})D_1^{(-1)} = 64(1-x^{32}) \quad \mbox{in $\Z G$}
\]
is given by
\[ 
D_1 = A_{10}(1+y^2)+A_{11}(1-y^2),
\]
where
\begin{align*}
A_{10} &= -\big((x+x^4+x^9+x^{12}+x^{14})(1+x^{16})+(x^6+x^7-x^{15})(1-x^{16})\big),\\
A_{11} &= -\big((x-x^9+x^{10})(1+x^{16})+(x^2+x^4-x^7+x^{12}-x^{15})(1-x^{16})\big)y.
\end{align*}
This was found by a difficult computer search.
Although a naive search for $D_1$ involves a search space of size $2^{64}$,
the search was shortened by using the irreducible representations induced by the character (homomorphism) that maps $x^{32}$ to $-1$, and by making some simplifying assumptions about the structure of the target difference set~\cite{yolland}.

Now $D_0(1+x^{32})$ has support 
$\big(1+x^2+x^3+x^5+(1+x^3+x^5+x^6)y\big)\langle x^{8}, y^2 \rangle$, 
and $D_1(1-x^{32})$ has support
$\big(x+x^4+x^6+x^7+(x+x^2+x^4+x^7)y\big) \langle x^{8}, y^2 \rangle$.
These supports are disjoint and their union is~$G$.
We conclude from the construction of \cref{prop:original-final} that 
$D= D_0(1+x^{32})+D_1(1-x^{32})$ corresponds to a difference set in~$G$.
\end{example}

We now reinterpret Examples~\ref{ex:modular} and~\ref{ex:yolland} as arising from a modification of a signature set.
\begin{lemma}
\label{lem:mod-sig}
Let $G$ be a group containing a normal subgroup $E \cong C_2^r$, and let $\{\chi_u: u \in U_r\}$ be the set of characters of~$E$. Let $A_u$ be a $\{+1,0,-1\}$-valued function on $G$ for each $u \in U_r$, where the $A_u$ have disjoint supports whose union is a set of coset representatives for $E_r$ in~$G$. Suppose that
\begin{equation}
\label{eqn:G22r}
\sum_{u \in U_r} A_u \chi_u A_u^{(-1)} = \frac{|G|}{2^{r}} \quad \mbox{in $\Z G$}.
\end{equation}
Then $G \in {\cal H}$.
\end{lemma}
\begin{proof}
Let 
\[
D=\sum_{u \in U_r}A_u \chi_u \quad \mbox{in $\Z G$}, 
\]
which by the assumption on the supports of the $A_u$ is a $\{\pm 1\}$-valued function on~$G$. 
We calculate $DD^{(-1)} = |G|$ using \cref{prop:orthogonality}~$(i)$, and so $D$ corresponds to a Hadamard difference set in~$G$.
\end{proof}
\noindent
By \cref{prop:orthogonality}~$(ii)$, one way to achieve \eqref{eqn:G22r} in \cref{lem:mod-sig} would be for the $A_u$ to satisfy the condition in $\Z G$ that
\begin{equation}
\label{eqn:strict}
A_u \chi_u A_u^{(-1)} = \tfrac{|G|}{2^{2r}} \chi_u \quad \mbox{for each $u \in U_r$}.
\end{equation}
Such a set of $A_u$ would be similar, but not identical, to a signature set on $G$ with respect to $E$: the conditions on the supports in \cref{lem:mod-sig} are different from those in \cref{defn:sig}, and the constant in \eqref{eqn:strict} is $\tfrac{|G|}{2^{2r}}$ rather than~$\tfrac{|G|}{2^r}$.

A crucial observation in reinterpreting Examples~\ref{ex:modular} and~\ref{ex:yolland} is that a weaker condition than \eqref{eqn:strict} suffices. In particular, in the case $r=2$, this condition can be weakened to
\begin{align}
A_{0j} \chi_{0j} A_{0j}^{(-1)} 
 &= \tfrac{|G|}{16} \chi_{0j} \quad \mbox{for each $j \in \{0,1\}$},  	\label{eqn:weak1} \\
A_{10} \chi_{10} A_{10}^{(-1)} + A_{11} \chi_{11} A_{11}^{(-1)} 	\label{eqn:weak2}
 &= \tfrac{|G|}{16} (\chi_{10}+\chi_{11}), 
\end{align}
in which the expressions $A_{10} \chi_{10} A_{10}^{(-1)}$ and $A_{11} \chi_{11} A_{11}^{(-1)}$
behave like a ``complementary pair'' whose sum is the same as if \eqref{eqn:strict} held.

In \cref{ex:modular}, the group $M_{64}$ contains the normal subgroup
$E_2=\langle x^{16}, y\rangle \cong C_2^2$ 
whose characters are
\[
\chi_{ij} = \big(1+(-1)^i x^{16}\big) \big(1+(-1)^j y\big) \quad \mbox{for $(i,j) \in U_2$}.
\]
The difference set $D$ takes the form
\[
D = D_0(1+x^{16}) + D_1(1-x^{16}) = \sum_{(i,j)\in U_2} A_{ij} \chi_{ij} 
\]
where the $A_{ij}$ take the values specified in the example.
These $A_{ij}$ satisfy the conditions of \cref{lem:mod-sig} on their supports. 
Since conjugation by $x$ fixes $\chi_{00}$ and~$\chi_{01}$ but swaps $\chi_{10}$ and~$\chi_{11}$, we find by direct calculation that 
\[
A_{0j} \chi_{0j} A_{0j}^{(-1)} = 4 \chi_{0j} \quad \mbox{for each $j \in \{0,1\}$}
\]
and 
\begin{align*}
\lefteqn{A_{10}\chi_{10}A_{10}^{(-1)} + A_{11}\chi_{11}A_{11}^{(-1)} } \hspace{20mm} \\
 &= \Big( 2(1-x^{-1})\chi_{10} + 2(1-x)\chi_{11} \Big) + 
    \Big( 2(1+x^{-1})\chi_{10} + 2(1+x)\chi_{11} \Big)  \\
 &= 4 (\chi_{10}+\chi_{11}),
\end{align*}
so that \eqref{eqn:weak1} and \eqref{eqn:weak2} hold. 

The reinterpretation of \cref{ex:yolland} is similar. SmallGroup$(256,536)$ contains the normal subgroup $E_2=\langle x^{32},y^2\rangle \cong C_2^2$, whose characters are 
\[
\chi_{ij} = \big(1+(-1)^i x^{32}\big) \big(1+(-1)^j y^2\big) \quad \mbox{for $(i,j) \in U_2$}.
\]
The difference set $D$ takes the form
\[
D = D_0(1+x^{32}) + D_1(1-x^{32}) = \sum_{(i,j)\in U_2} A_{ij} \chi_{ij},
\]
where the $A_{ij}$ take the values specified in the example.
These $A_{ij}$ satisfy the conditions of \cref{lem:mod-sig} on their supports. 
Conjugation by $x$ fixes $\chi_{00}$ and~$\chi_{01}$ but swaps $\chi_{10}$ and~$\chi_{11}$, and we find once again (after a long calculation) that \eqref{eqn:weak1} and \eqref{eqn:weak2} hold.

\subsection{Combination of Signature Sets and Perfect Ternary Arrays}
\label{subsec:combination}

The nonabelian signature set approach of \cref{subsec:nonabelian-sigset} and the perfect ternary array product construction of \cref{subsec:productPTA} are closely related.
For example, \cref{prop:HDSbyC2} may be interpreted as constructing a signature set on $H \times E_1$ from a perfect ternary array $D$ in $H$.
We now illustrate how a perfect ternary array in a factor group can be used to create a signature block with respect to a specific character.
We believe the illustrated method could be useful in future studies of the existence pattern for Hadamard difference sets in 2-groups of order greater than~$256$.

\begin{lemma}
\label{lem:PTAtoSS}
Let $K$ be a group containing a central subgroup $E \cong C_2^r$, and let $\chi$ be a character of~$E$.
Suppose that $\chi = H\chi'$ in $\Z E$ for some subgroup~$H$ of~$E$.
Let $\natural$ be the natural map from $K$ onto $K/H$, and suppose that $A$ is a $\{+1, 0,-1\}$-valued function on $K$ for which $\natural(A)$ is a perfect ternary array of weight $2^{2j}$ in~$K/H$. Then
\[
A \chi A^{(-1)} = 2^{2j}\chi \quad \mbox{in $\Z K$}. 
\]
\end{lemma}

\begin{proof}
Since $\natural(A)$ is a perfect ternary array of weight $2^{2j}$ in $K/H$, in $\Z(K/H)$ we have by \cref{lem:PTAwt} that
\[
2^{2j} 1_{K/H} = \natural(A) \natural(A)^{(-1)} 
= (AH)(A^{(-1)}H) = AA^{(-1)}H.
\]
For $k \in K$, interpret the element $kH$ in $K/H$ as $|H|$ elements in $K$, so that in the group ring $\Z K$ the above equation becomes
\[
2^{2j} H = AA^{(-1)} H.
\]
By assumption we have $\chi = H\chi'$, and $H$ and $\chi'$ are central in $K$ because $E$ is.
Therefore in $\Z K$ we have
\[
A\chi A^{(-1)}=AH\chi'A^{(-1)} = AA^{(-1)}H\chi' = 2^{2j}H\chi' = 2^{2j}\chi.
\]
\end{proof}
\noindent
In \cref{lem:PTAtoSS}, note that the group ring condition $\chi = H\chi'$ is equivalent to $H \in \Ker(\chi)$ when the character $\chi$ is considered as a homomorphism of~$E$.
Also note that if $E$ has index $2^{2j}$ in $K$,
and $A$ is $\{\pm 1\}$-valued on a set of coset representatives for $E$ in $K$, 
then the conclusion of \cref{lem:PTAtoSS} is that $A$ is a signature block on $K$ with respect to~$\chi$.

We now use \cref{lem:PTAtoSS} to explain the origin of the signature set introduced in \cref{ex:C82}.

\begin{example}
\label{ex:C82revisit}
Let $K = \langle X,Y \rangle \cong C_4^2$ and $E = \langle X^2, Y^2 \rangle \cong C_2^2$, and let $\{\chi_u: u \in U_2\}$ be the set of characters of~$E$.
We use \cref{lem:PTAtoSS} to construct the signature set 
\[
A_{00} = A_{01} = A_{10} = 1 + X + Y - X Y \quad \mbox{and} \quad A_{11} = 1 + X + Y + X Y
\]
on $K$ that was presented in \cref{ex:C82} without explanation of its origin.

For $\chi = \chi_{00}$ or $\chi_{10}$, take $H = \langle Y^2 \rangle$ and $A = 1-X-Y-XY$.
Then $\natural(A)$ is a perfect ternary array of weight $4$ in $K/H$ by \cref{ex:pta-mod2}~(i), because $\natural(Y)$ is an involution that commutes with the nonidentity element $\natural(X)$.
\cref{lem:PTAtoSS} then shows that $A$ is a signature block on $K$ with respect to~$\chi_{00}$ and~$\chi_{10}$. Since $A_{00} \chi_{00} = -XYA \chi_{00}$ and $A_{10} \chi_{10} = XA \chi_{10}$ in $\Z K$, it follows from \cref{defn:sig} and \cref{prop:orthogonality}~(i) that $A_{00} = A_{10}$ is a signature block on $K$ with respect to both $\chi_{00}$ and~$\chi_{10}$.
By symmetry in $X$ and $Y$, it follows that $A_{01}$ is also a signature block on $K$ with respect to~$\chi_{01}$.

For $\chi = \chi_{11}$, take $H = \langle X^2Y^2 \rangle$ and $A= 1+X+XY-X^2Y$. 
Then $\natural(A)$ is a perfect ternary array of weight $4$ in $K/H$ by \cref{ex:pta-mod2}~(i), because $\natural(XY)$ is an involution that commutes with the nonidentity element~$\natural(X)$.
By \cref{lem:PTAtoSS} and the relation $A_{11} \chi_{11} = A \chi_{11}$ in $\Z K$, 
we conclude that $A_{11}$ is a signature block on $K$ with respect to~$\chi_{11}$.
\end{example}

\section{Computer Implementation for Groups of Order 256}
\label{sec:verification}

In this section, we provide further details of the streamlined procedure used to establish that each of the 56,049 groups of order 256 not excluded by Theorems~\ref{thm:turyn} and~\ref{thm:dillon} belongs to ${\cal H}$. We then describe online databases containing difference sets found by this procedure, and explain how the overall result can be quickly verified on a desktop computer using the accepted GAP package \emph{DifSets}~\cite{peifer,difsets}. We note that \emph{DifSets} provides (via the LoadDifferenceSets command) a listing of all inequivalent difference sets in groups of order~$16$ and~$64$.

\subsection{Procedure}
As previously summarized in \cref{sec:intro}, the streamlined procedure for groups of order 256 comprises three stages:
\begin{description}
\item[Stage 1.]
Use \cref{thm:main} to account for the 54,633 groups containing a normal subgroup isomorphic to $C_2^4$ or $C_4^2 \times C_2$ or $C_8^2$.

\item[Stage 2.]
Use the product construction of \cref{prop:productPTAgivesDS} to account for $1,358$ further groups.

\item[Stage 3.]
Apply the transfer methods of \cref{subsec:transfer} to account for 57 further groups, and the modified signature set method of \cref{subsec:final} to account for the final group. We do not describe this stage further.
\end{description}
\noindent
The relationship between the groups handled by Stages 1 and 2 is shown in \cref{fig:two256constructions}.

\begin{figure}[h]
\hspace{5.0em}
\begin{tikzpicture}
\draw[black,thick] (0,   1.6) rectangle (9.8,6.8);
\draw[blue,thick]  (3.9, 3.8) circle [radius=1.7];
\draw[red,thick]   (5.9, 3.8) circle [radius=1.7];
\node [align=center]      at (11.5,4.1) {56,049 non-excluded\\groups of order 256};
\node [align=center,blue] at (1.6, 5.9) {54,633 groups\\accounted for by\\\cref{thm:main}};
\node [align=center,red]  at (8.2, 5.9) {55,330 groups\\accounted for by\\\cref{prop:productPTAgivesDS}};
\node [align=center] at (3.2,3.85) {661};
\node [align=center] at (4.9,3.83) {53,972};
\node [align=center] at (6.6,3.85) {1,358};
\node [align=center] at (1.4,3.1) {58};
\end{tikzpicture}
\caption{\cref{thm:main} and~\cref{prop:productPTAgivesDS} show that at most 58 of the 56,049 non-excluded groups of order 256 lie outside~${\cal H}$.}
\label{fig:two256constructions}
\end{figure}

In Stage 1, we wish to construct a difference set in a group $G$ of order $256$ containing a normal abelian subgroup $K$, where $K$ is isomorphic to $C_2^4$ or $C_4^2 \times C_2$ or $C_8^2$.
A signature set on $K$ is provided trivially for the case $C_2^4$ (see the remark following \cref{defn:sig}),
by \cref{ex:sigset844} for the case $C_4^2 \times C_2$,
and by \cref{ex:rank2} for the case $C_8^2$.
We then apply the method in the proof of \cref{thm:prehadamard} to construct a difference set in $G$. This requires a set $\{g_u : u \in U_r\}$ of coset representatives for $K$ in $G$ satisfying \eqref{eqn:permute-chiu3}, namely
\[
\{g_u \chi_u g_u^{-1} : u \in U_r\} = \{\chi_u: u \in U_r\}.
\]
The existence of such a set is guaranteed by \cref{thm:drisko}, but the proof of this result in \cite{drisko} is non-constructive. We therefore conduct a search for a suitable set of coset representatives~$\{g_u\}$. This search is exhaustive for the cases $C_4^2 \times C_2$ and~$C_8^2$, but random for the case $C_2^4$ whose search space has size $15! > 10^{12}$.

The results of applying this search procedure to all 56,049 non-excluded groups, for each of the three choices of $K$ independently, are shown in \cref{fig:mainthm256constructions}.

\begin{figure}[h]
\hspace{4.0em}
\begin{tikzpicture}
\draw[black,thick] (0,   1.5) rectangle (11.5,9.2);
\draw[blue,thick]  (4.5, 5.7) circle [radius=2];
\draw[red,thick]   (6.5, 5.7) circle [radius=2];
\draw[mygreen,thick] (5.5, 3.9) circle [radius=2];
\node [align=center] at (13.2,5.5) {56,049 non-excluded\\groups of order 256};
\node [align=center,blue] at (2.0,8.2) {42,268 groups contain\\a normal subgroup\\isomorphic to $C_2^4$};
\node [align=center,red] at (9.5,8.2) {49,165 groups contain\\a normal subgroup\\isomorphic to $C_4^2 \times C_2$};
\node [align=center,mygreen] at (9.1,2.6) {684 groups contain\\a normal subgroup\\isomorphic to $C_8^2$};
\node [align=center] at (3.5,6.0) {5,132};
\node [align=center] at (5.5,6.4) {37,120};
\node [align=center] at (7.5,6.0) {11,697};
\node [align=center] at (5.5,5.05) {16};
\node [align=center] at (5.5,2.9) {336};
\node [align=center] at (4.3,4.4) {0};
\node [align=center] at (6.7,4.4) {332};
\node [align=center] at (1.5,3.4) {1,416};
\end{tikzpicture}
\caption{\cref{thm:main} shows that at most 1,416 of the 56,049 non-excluded groups of order 256 lie outside~${\cal H}$.}
\label{fig:mainthm256constructions}
\end{figure}
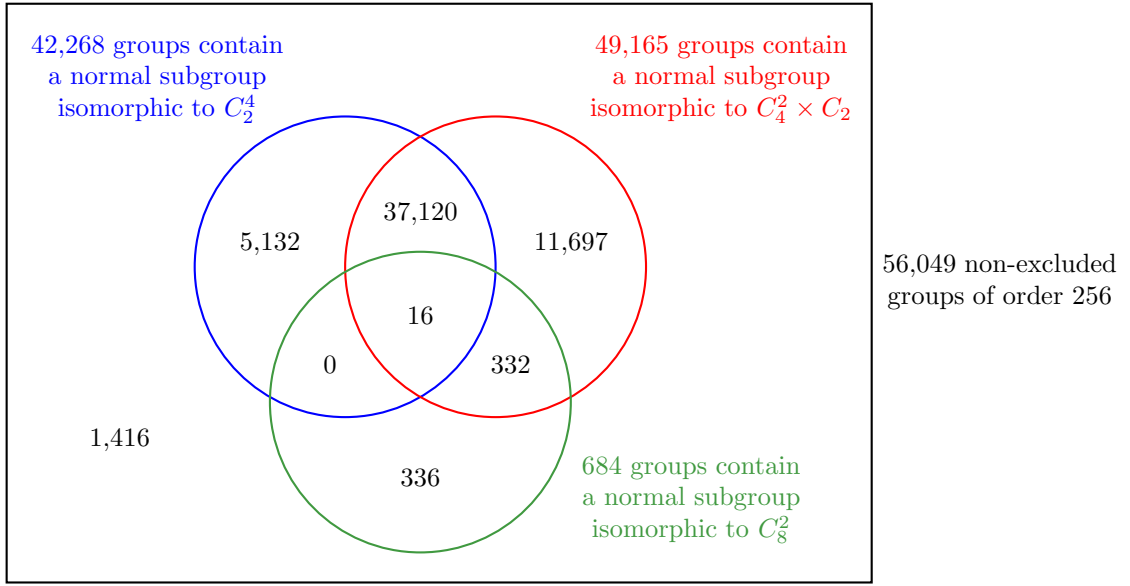

In Stage 2, we distinguish six instances of the product construction of \cref{prop:productPTAgivesDS} according to the form of its input perfect ternary arrays $T_1, T_2, \dots, T_s$.

\begin{enumerate}[$(i)$]
\item
$\mathbf{H_{64}\cdot Q_4}$ {\bf form}.
Take $T_1$ to be a Hadamard difference set in a subgroup $H_1$ of $G$ of order~$64$, and $T_2$ to be a perfect ternary array of weight $4$ in~$G$ having the form of \cref{ex:pta-mod2}~$(ii)$ where the quaternion group $Q = \langle x, y \rangle$ of order $8$ intersects $H_1$ in the two-element subgroup $\{1, x^2\}$.

\item
$\mathbf{H_{64}\cdot H_4}$ {\bf form}.
Take $T_1$ to be a Hadamard difference set in a subgroup $H_1$ of $G$ of order~$64$, and $T_2$ to be a Hadamard difference set in a subgroup $H_2$ of $G$ of order~$4$, where $G=H_1H_2$ and $H_1 \cap H_2 = 1$.

\item
$\mathbf{H_{16}\cdot H_{16}}$ {\bf form}.
For $i=1,2$, take $T_i$ to be a Hadamard difference set in a subgroup $H_i$ of $G$ of order~$16$, where $G=H_1H_2$ and $H_1 \cap H_2 = 1$.

\item
$\mathbf{H_{64}\cdot T_1}$ {\bf form}.
Take $T_1$ to be a perfect ternary array of weight $4$ in~$G$ having the form of \cref{ex:pta-mod2}~$(i)$, and $T_2$ to be a Hadamard difference set in a subgroup of $G$ of order~$64$.

\item
$\mathbf{H_{16}\cdot T_1 \cdot T_2}$ {\bf form}.
Take each of $T_1,T_2$ to be a perfect ternary array of weight $4$ in~$G$ having either of the two forms of \cref{ex:pta-mod2}, and $T_3$ to be a Hadamard difference set in a subgroup of $G$ of order~$16$.

\item
$\mathbf{T_1 \cdot T_2 \cdot T_3 \cdot T_4}$ {\bf form}.
Take each of $T_1, T_2, T_3, T_4$ to be a perfect ternary array of weight $4$ in~$G$ having either of the two forms of \cref{ex:pta-mod2}.

\end{enumerate}

For each of these six forms, we conduct a search for a suitable set of perfect ternary arrays satisfying all the required conditions. 
The search for the forms $(i)$ to $(iii)$ is relatively fast because the search is restricted to subgroups of the appropriate order. However, the search for the forms $(iv)$ to $(vi)$ is not constrained in this way and can take considerably longer; the search for form $(vi)$ sometimes requires more than a day for a single group.

We therefore begin by searching all $56,049$ non-excluded groups for each of the forms $(i)$ to $(iii)$ independently, with results as shown in \cref{fig:mainthm256constructions}. 
We then conduct a search for each of the forms $(iv)$ to $(vi)$ in that order, but only over those groups in which no previous form has been found. The number of groups accounted for and remaining at each step of Stage 2 is shown below.
\begin{center}
\begin{tabular}{|l|c|c|c|c|}
				 		 		 	   \hline
			& $H_{64} \cdot Q_4$ 		& $H_{64} \cdot T_1$ 	& $H_{16} \cdot T_1 \cdot T_2$	& $T_1 \cdot T_2 \cdot T_3 \cdot T_4$ \\ 
			& and $H_{64} \cdot H_4$ 	&			&				&	\\ 
			& and $H_{16} \cdot H_{16}$ 	&			&				&	\\ \hline
\# groups accounted for	& 51,957			& 3,119			& 236				& 18	\\ \hline
\# groups remaining 	& 4,092				& 973			& 737				& 719	\\ \hline
\end{tabular}
\end{center}
The Stage 2 searches are exhaustive in that none of the remaining 719 groups contains a difference set having one of the six forms $(i)$ to~$(vi)$.

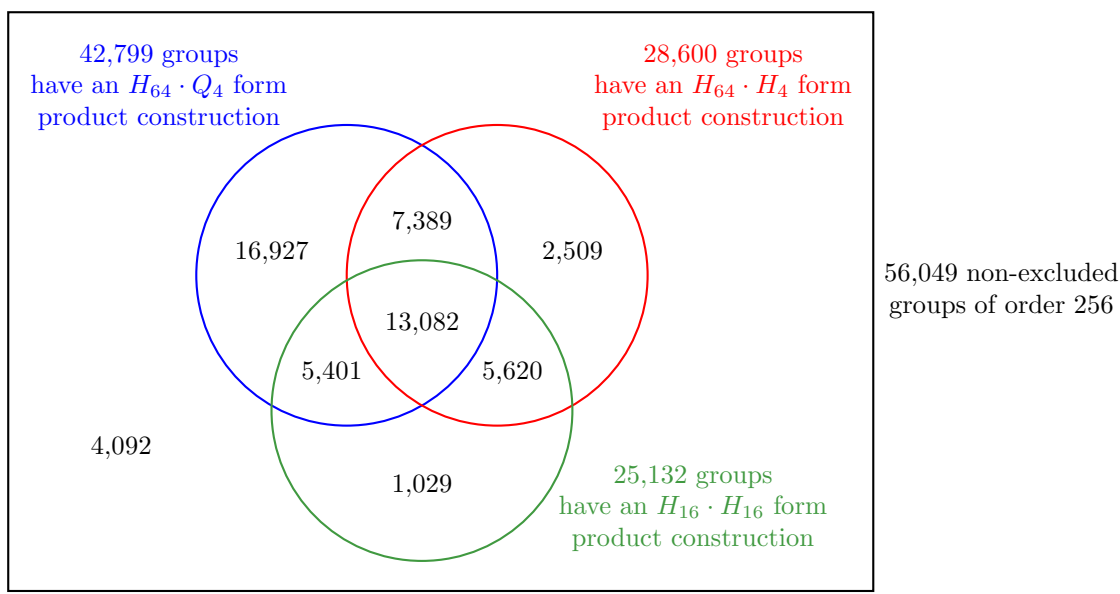
\begin{figure}[h]
\hspace{3.8em}
\begin{tikzpicture}
\draw[black,thick] (0,   1.5) rectangle (11.5,9.2);
\draw[blue,thick]  (4.5, 5.7) circle [radius=2];
\draw[red,thick]   (6.5, 5.7) circle [radius=2];
\draw[mygreen,thick] (5.5, 3.9) circle [radius=2];
\node [align=center] at (13.2,5.5) {56,049 non-excluded\\groups of order 256};
\node [align=center,blue] at (2.0,8.2) {42,799 groups\\have an $H_{64} \cdot Q_4$ form\\product construction};
\node [align=center,red] at (9.5,8.2) {28,600 groups\\have an $H_{64} \cdot H_4$ form\\product construction};
\node [align=center,mygreen] at (9.1,2.6) {25,132 groups\\have an $H_{16} \cdot H_{16}$ form\\product construction};
\node [align=center] at (3.5,6.0) {16,927};
\node [align=center] at (5.5,6.4) {7,389};
\node [align=center] at (7.5,6.0) {2,509};
\node [align=center] at (5.5,5.05) {13,082};
\node [align=center] at (5.5,2.9) {1,029};
\node [align=center] at (4.3,4.4) {5,401};
\node [align=center] at (6.7,4.4) {5,620};
\node [align=center] at (1.5,3.4) {4,092};
\end{tikzpicture}
\caption{Forms $H_{64} \cdot Q_4$ and $H_{64} \cdot H_4$ and $H_{16} \cdot H_{16}$ of \cref{prop:productPTAgivesDS} show that at most 4,092 of the 56,049 non-excluded groups of order 256 lie outside~${\cal H}$.}
\label{fig:product256constructions}
\end{figure}

\subsection{Databases and Verification}
The website \cite{smith-online} contains ten lists in GAP format, organized into four databases as shown in \cref{tab:databases}.

\begin{table}[h]
\begin{center}
\begin{tabular}{|c|l|l|}
				 		 		 	   	   \hline
List 	& List name 			& Database name				\\ \hline
$L_1$	& HDS256\_Normal\_02x02x02x02	&					\\
$L_2$	& HDS256\_Normal\_04x04x02	& HDS256\_NormalSubgroupTransversal.txt \\
$L_3$	& HDS256\_Normal\_08x08		&					\\ \hline
$L_4$	& HDS256\_H64byQ4		&					\\
$L_5$	& HDS256\_H64byH4		& HDS256\_PTAProduct.txt 		\\
$L_6$	& HDS256\_H16byH16		&					\\ \hline
$L_7$	& HDS256\_H64byT1		&					\\
$L_8$	& HDS256\_H16byT1byT2		& HDS256\_SubgroupProduct.txt		\\
$L_9$	& HDS256\_T1byT2byT3byT4	&					\\ \hline
$L_{10}$& HDS256			& HDS256.txt 				\\ \hline
\end{tabular}
\end{center}
\caption{Organization of difference set databases in \cite{smith-online}.}
\label{tab:databases}
\end{table}

Lists $L_1$ to $L_3$ correspond to the three circles in \cref{fig:mainthm256constructions} (Stage 1).
Lists $L_4$ to $L_6$ correspond to the three circles in \cref{fig:product256constructions} (forms $(i)$ to $(iii)$ of Stage~2).
Lists $L_7$ to $L_9$ correspond to forms $(iv)$ to $(vi)$ of Stage~2.
Each entry of the lists $L_1$ to $L_9$ contains at least two fields: a catalogue number $i$ that identifies the group SmallGroup$(256, i)$, and a list of 120 indices taken from $\{1,2,\dots,256\}$ in which index $j$ labels group element~$j$ according to the GAP ordering given by Elements(SmallGroup$(256, i)$).

The list $L_{10}$ contains one entry for each of the 56,092 groups of order~$256$. If 
SmallGroup$(256,i)$ is one of the 43 groups excluded by Theorems~\ref{thm:turyn} and~\ref{thm:dillon} (see \cref{tab:2-groups}), then entry $i$ of $L_{10}$ is an emtpy list of indices.
Otherwise, this entry is a list of 120 indices corresponding to a representative difference set in SmallGroup$(256,i)$. The representative difference set is taken from list $L_1$ if possible, otherwise from $L_2$, and so on to $L_9$. 
This accounts for the origin of all but 58 of the non-empty entries of~$L_{10}$.

After reading the list HDS256 into the current directory, the following GAP code uses Peifer's accepted GAP package \emph{DifSets} \cite{difsets} to verify that HDS256 contains an index list corresponding to a difference set for 56,049 groups of the 56,092 groups of order 256, and an empty index list for the remaining 43 groups. 

\begin{verbatim}
LoadPackage("DifSets");
empty := 0;
count := 0;
for i in [1..Length(HDS256)] do;
  if HDS256[i] = [] then
    empty := empty+1;
  else
    if IsDifferenceSet(SmallGroup(256,i), HDS256[i]) then
      count := count+1;
    fi;
  fi;
od;
Print("HDS256 contains ", Length(HDS256), " index lists, of which\n");
Print(count, " correspond to a difference set and ", empty, " are empty\n");
\end{verbatim}

It took less than $20$ minutes to run this code on a 2013 iMac desktop computer using a standard implementation of GAP, producing the following output. 

\begin{verbatim}
HDS256 contains 56092 index lists, of which
56049 correspond to a difference set and 43 are empty
\end{verbatim}

Although we found it considerably more difficult to construct a difference set in some groups of order 256 than in others, there is no significant variation in verification time among groups of a given order using the IsDifferenceSet command of \emph{DifSets}.

\section{Future Directions}
\label{sec:future}
In this section, we propose directions for future research into Hadamard difference sets and their relations to other combinatorial objects.

We have described in this paper a streamlined procedure for demonstrating that all groups of order 64 and 256, apart from those that are excluded by the classical nonexistence results of Theorems~\ref{thm:turyn} and~\ref{thm:dillon}, belong to the class~${\cal H}$ of Hadamard difference sets. While we consider this to be a major achievement in combinatorics, it is unsatisfactory that we do not yet have a completely theoretical demonstration.

We now propose the following directions for future research into Hadamard difference sets, with three overall objectives in mind. The first objective is to simplify and unify the various techniques of \cref{sec:further}, removing the reliance on extensive computer search and the non-systematic transfer methods. The second objective is to develop recursive or direct construction techniques for nonabelian groups, that are as powerful as \cref{thm:abelian-sigset} is for constructing signature sets on abelian groups. The third and ultimate objective is to resolve \cref{quest:the-big-one}.

\begin{enumerate}[D1.]

\item
The concept of signature sets on abelian groups (\cref{thm:abelian-sigset}) and on nonabelian groups (\cref{subsec:nonabelian-sigset}) appears to be very powerful. Develop construction methods to determine all nonabelian groups on which there is a signature set relative to a normal elementary abelian subgroup.

\item
Apply \cref{lem:PTAtoSS} to create signature sets in nonabelian groups, generalizing the model of \cref{ex:C82revisit}.

\item 
Understand when and why the transfer methods of \cref{subsec:transfer} succeed.

\item
Develop a general theory based on the method of \cref{subsec:final} so that transfer methods are no longer needed for groups of order 64 and~256.

\item
Representation theory was used to help find the group ring element $D_1$ in Examples \ref{ex:modular} and~\ref{ex:yolland}. Apply representation theory to unify and extend the construction methods of \cref{sec:further}.

\item
In the study of bent functions, which are equivalent to Hadamard difference sets in elementary abelian $2$-groups, one asks how many inequivalent examples exist in a given group.
Determine how many inequivalent Hadamard difference sets in (not necessarily elementary abelian) $2$-groups can be constructed using the methods of this paper.

\item 
Formulate a theoretical framework that can be systematically applied to determine all $2$-groups belonging to~${\cal H}$.

\item 
Extend the transfer methods of \cref{subsec:transfer} to construct Hadamard difference sets in new groups whose order is not a power of 2, for example in groups of order 100 \cite{golemac-vucicic}, 144 \cite{vucicic}, or 400 \cite{mandic-vucicic}.
\end{enumerate}

We also propose some further research directions involving the relation of Hadamard difference sets to other combinatorial objects.

\begin{enumerate}[D1.]

\item[D9.]
Difference sets in the Hadamard, McFarland, Spence, and Chen-Davis-Jedwab families have parameters $(v,k,\lambda)$ satisfying $\gcd(v,k-\lambda) > 1$, and are known to share construction methods based on covering extended building sets and semi-regular relative difference sets \cite{unifying, chen}. Adapt the signature set approach for Hadamard difference sets in order to construct difference sets in nonabelian groups for the other three families, and the associated semi-regular relative difference sets in nonabelian groups for all four families.

\item[D10.]
Determine how many inequivalent designs arise from the Hadamard difference sets constructed in this paper.

\item[D11.]
Determine how many inequivalent binary codes arise from the incidence matrices of the Hadamard difference sets constructed in this paper.

\end{enumerate}

\end{document}